\documentclass{jgcc}

\pdfoutput=1
\usepackage{lastpage}
\jgccdoi{18}{1}{3}{17585}
\jgccheading{}{\pageref{LastPage}}{}{}{Feb.~26,~2026}{Apr.~1,~2026}{}

\usepackage[utf8]{inputenc}
\usepackage{ifthen}

\usepackage{xspace}
\usepackage{graphicx}

\usepackage{multirow}

\usepackage{tcolorbox}
\usepackage{dsfont}
\usepackage{textcomp}
\usepackage{bold-extra}
\usepackage{verbatim,wrapfig}

\usepackage{amsmath,amssymb,amsthm}
\usepackage{mathtools}
\usepackage{bm}
\usepackage{MnSymbol}

\usepackage{tikz}
\usetikzlibrary{shapes}
\usetikzlibrary{arrows,automata,positioning}
\usetikzlibrary{decorations.pathreplacing}

\newcommand{\Sol}{\mathrm{Sol}}

\newcommand{\vdmatrix}[4]{\left(\begin{smallmatrix}#1 & #2\\ #3 & #4\end{smallmatrix}\right)}

\newcommand\fsqrtp{finite square roots property\xspace}

\renewcommand{\ast}{*}
\renewcommand{\colon}{:}

\renewcommand{\iff}{\mathrel{\Leftrightarrow}}

\DeclareMathOperator{\nf}{\mathrm{nf}}
\DeclareMathOperator{\nfslex}{\mathrm{nf}_{{\mathrm{slex}}}}

\newcommand\fg{f.g.\xspace}

\newcommand{\wrt}{wrt.\xspace}

\newcommand{\lds}{, \ldots ,}

\newcommand{\GP}{\mathop{\text{GP}}}

\newcommand{\Mati}{Matiyasevich\xspace}

\newcommand{\IFF}{if and only if\xspace}
\renewcommand{\hom}{homo\-mor\-phism\xspace}

\newcommand{\iso}{isomor\-phism\xspace}

\newcommand{\epi}{epimor\-phism\xspace}
\newcommand{\epis}{epimor\-phisms\xspace}

\newcommand{\mazu}{Mazurkiewicz\xspace}

\DeclareMathOperator{\pos}{\mathrm{pos}}

\newcommand{\tra}{transition\xspace}

\newcommand{\eg}{e.g.\xspace}

\newcommand{\Ip}{In parti\-cu\-lar,\xspace}

\newcommand{\invol}{involution\xspace}

\newcommand{\solu}{solu\-tion\xspace}
\newcommand{\solus}{solu\-tions\xspace}

\usepackage{prettyref}
\newcommand{\prref}[1]{\prettyref{#1}}
\newrefformat{thm}{Theorem~\ref{#1}}
\newrefformat{lem}{Lemma~\ref{#1}}
\newrefformat{def}{Definition~\ref{#1}}
\newrefformat{claim}{Claim~\ref{#1}}
\newrefformat{conj}{Conjecture~\ref{#1}}
\newrefformat{ques}{Question~\ref{#1}}
\newrefformat{cor}{Corollary~\ref{#1}}
\newrefformat{inva}{Invariant~\ref{#1}}
\newrefformat{prop}{Proposition~\ref{#1}}
\newrefformat{sec}{Section~\ref{#1}}
\newrefformat{kap}{Chapter~\ref{#1}}
\newrefformat{ex}{Example~\ref{#1}}
\newrefformat{eq}{Equation~(\ref{#1})}
\newrefformat{rem}{Remark~\ref{#1}}
\newrefformat{fig}{Figure~\ref{#1}}
\newrefformat{par}{Paragraph~\ref{#1}}
\newrefformat{obs}{Observation~\ref{#1}}
\newrefformat{pro}{Property~\ref{#1}}
\newrefformat{pm}{Problem~\ref{#1}}
\newrefformat{ruli}{Rule~\ref{#1}}
\newrefformat{ruls}{Rules~\ref{#1}}

\newcommand{\arc}[1]{\overset{#1}\ra}

\newcommand{\set}[2]{\{#1\mid #2\}}

\newcommand{\os}[1]{\{#1\}}
\newcommand{\sm}{\setminus}
\newcommand{\es}{\emptyset}
\newcommand{\sse}{\subseteq}

\newcommand{\abs}[1]{|{#1}|}
\newcommand{\Abs}[1]{\Vert{#1}\Vert}

\newcommand{\M}{\ensuremath{\mathbb{M}}}

\newcommand{\N}{\ensuremath{\mathbb{N}}}
\newcommand{\Z}{\ensuremath{\mathbb{Z}}}
\newcommand{\Q}{\ensuremath{\mathbb{Q}}}

\newcommand{\G}{\ensuremath{\mathbb{G}}}

\newcommand{\id}{{\mathrm{id}}}

\renewcommand{\phi}{\varphi}
\newcommand{\eps}{\varepsilon}

\newcommand{\alp}{\alpha}

\newcommand{\gam}{\gamma}
\newcommand{\del}{\delta}
\newcommand{\lam}{\lambda}

\newcommand{\sig}{\sigma}
\newcommand{\Sig}{\Sigma}
\newcommand{\Gam}{\Gamma}

\newcommand{\Del}{\Delta}

\newcommand{\cA}{\mathcal{A}}

\newcommand{\cG}{\mathcal{G}}

\newcommand{\cL}{\mathcal{L}}

\newcommand{\cS}{\mathcal{S}}

\newcommand{\cX}{\mathcal{X}}

\newcommand{\ra}{\rightarrow}

\newcommand{\ov}[1]{\overline{#1}}
\newcommand{\oi}[1]{{#1}^{-1}}

\newcommand{\Rat}{\mathop{\mathrm{Rat}}}
\newcommand{\Rec}{\mathop{\mathrm{Rec}}}
\newcommand{\Reg}{\mathop{\mathrm{Reg}}}

\newcommand{\nflex}{\mathop{\mathrm{nf_{\mathrm{lex}}}}}

\newcommand{\nfM}{\mathrm{nf}_{\M}}

\newcommand{\HD}{Hasse diagram\xspace}
\newcommand{\HA}{Hasse arc\xspace}

\newcommand{\MGI}{\ensuremath{{M}(\Gam,I)}} 
\newcommand{\MSI}{\ensuremath{M(\Sig,I)}} 

\newcommand{\MJI}{\ensuremath{M(J,I)}}
\newcommand{\GGI}{\ensuremath{{G}(\Gam,I)}}

\newcommand\CoS[1]{{#1}^{\text{co-}*}}
\newcommand{\BS}{\ensuremath{\mathop\mathrm{BS}}}
\newcommand{\BSG}{Baumslag-Solitar group\xspace}

\DeclareMathOperator{\ts}{ts}
\DeclareMathOperator{\TO}{TO}
\DeclareMathOperator{\Pref}{Pref}
\DeclareMathOperator{\Suf}{Suf}

\begin{document}

\title[Quadratic equations, graph products, and the exponent of periodicity]{Quadratic equations in graph products of groups\\ and the exponent of periodicity}

\author[V.~Diekert]{Volker Diekert\,\orcidlink{0000-0002-5994-3762}}
\address{University of Stuttgart, Germany}
\email{diekert@fmi.uni-stuttgart.de}

\author[S.~Natterer]{Silas Natterer\,\orcidlink{0009-0004-6524-0854}}
\address{University of Stuttgart, Germany}
\email{natterersilas@gmail.com}

\author[A.~Thumm]{Alexander Thumm\,\orcidlink{0009-0005-4240-2045}}
\address{University of Siegen, Germany}
\email{alexander.thumm@uni-siegen.de}

\thanks{\textit{2020~Mathematics Subject Classification.} Primary 20F70; Secondary 68Q45.}


\begin{abstract}
In 1977, Makanin established the decidability of equations in free monoids. 
A key ingredient in his proof is the \emph{exponent of periodicity}: for a word $w$, it is the largest exponent $e$ such that $w$ contains a nonempty factor of the form $p^e$.
Makanin showed the following for a system of equations in free monoids: 
if the system has a solution with a sufficiently large exponent of periodicity, then it has infinitely many solutions.
However, the converse~--~whether the existence of infinitely many solutions implies the existence of solutions with arbitrarily large exponent of periodicity~--~remains open.

In this paper, we investigate the analogous problem for quadratic equations in finitely generated groups. 
We use normal forms to define the exponent of periodicity. 
We then identify structural conditions on groups and their normal forms that guarantee that infinite solution sets of quadratic systems have an unbounded exponent of periodicity. 
We prove that these conditions are preserved under graph products and, in particular, hold for all finitely generated right-angled Artin groups. 
In addition, we show that they also hold for finitely generated (graph products of) torsion-free nilpotent and hyperbolic groups, and we characterize the Baumslag-Solitar groups satisfying them.
\end{abstract}

\maketitle

\hfill{{\small \it Dedicated to Alexei Miasnikov}}\par
\hfill{{\small \it on the occasion of his birthday.}}

\section{Introduction}\label{sec:intro}

All formal definitions and results mentioned in the introduction are given or repeated in the main body of the paper. Various statements in the main body are formulated with respect to constraints. For better reading, constraints are not discussed in this introductory section. 

\subsection{A Brief History of Word Equations}\label{sec:hawkin}
We begin this excursion not with Diophantus of Alexandria, but more than~1700 years later.\footnote{Ceci n'est pas une publicit\'e pour \cite{Diekert2015CAI} comme l'aurait dit Ren\'e Magritte (1898--1967).} 
In the mid-1960s, it was known that the satisfiability problem for word equations in free semigroups reduces to Hilbert's Tenth Problem.
But this track was not used when \Mati~\cite{Matijasevic70} showed in 1970 that Hilbert's Tenth Problem is undecidable, using previous results by Robinson, Davis, and Putnam.\footnote{For a brief history of Hilbert's Tenth Problem, we refer to \cite{MMatiyasevichOPSW}.} 

The breakthrough regarding the satisfiability problem for word equations in free monoids and semigroups is due to Makanin. He showed in 1977 in his seminal paper \cite{mak77} that the existential theory in free semigroups is indeed decidable. In Makanin's proof (of about 80 dense pages), the notion of the \emph{exponent of periodicity} plays a central role. The exponent of periodicity $\exp(w)$ of a word~$w$ is the largest natural number~$e$ such that we can write $w=up^{e}v$ where $p$ is nonempty.\footnote{This implies that $p$ is a primitive word.} 
For a set $L$ of words, the exponent of periodicity, denoted by $\exp(L)$, is defined as the supremum $\sup\set{\exp(w)}{w \in L} \in \N \cup \os{\infty}$. 

The interest here is in the exponent of periodicity of the solution set $\Sol(\cS)$ of a given finite system of word equations $\cS$.
Let $\Gam$ and $\cX$ be finite sets of \emph{constants} and \emph{variables}, respectively.
Then $\cS$ is a set of equations $U=V$ where $U, V\in (\Gam\cup \cX)^*$, 
and a \emph{\solu} of $\cS$ is given by a mapping $\sigma:\mathcal{X}\to \Gamma^*$ such that substituting every occurrence of every $X \in \cX$ by its corresponding image $\sigma(X)\in \Gamma^*$ yields an equality $\sigma(U)=\sigma(V)\in \Gamma^*$ for all $(U=V)\in \cS$.
We then define the exponent of periodicity of a solution $\sigma \in \Sol(\cS)$ and of the system of word equations itself by $\exp(\sigma) = \exp(\sigma(\cX))$ and $\exp(\cS) = \sup\set{\exp(\sigma)}{\sigma \in \Sol(\cS)}$, respectively.

In his paper, Makanin crucially used a result of Hmelevski\u{\i} \cite{hme76} that the exponent of periodicity $\exp(\sigma)$ of a shortest \solu $\sigma:\mathcal{X}\to \Gamma^*$ of a finite system word equation $\cS$ is bounded by a computable function $f$ in the input size. 
Moreover, if there is any $\sigma' \in \Sol(\cS)$ with a higher exponent than $f(n)$, where $n$ is the input size of $\cS$, then there are infinitely many \solus and, in addition, we have $\exp(\cS)=\infty$. The amazing fact is that we do not know whether the converse holds.

\begin{prob}\label{pm:hugo}
Does the existence of infinitely many 
solutions for a finite system of word equations $\cS$ in a free semigroup imply $\exp(\cS)=\infty$?
\end{prob}

We believe that \prref{pm:hugo} has a positive answer, but so far we have only been able to prove this conjecture in restricted cases. The state of the art (as it is known to the authors to date) is reported in our paper \cite{DiekertNT2026arxiv}, which provides a positive result for quadratic equations (with certain regular constraints) and for equations with at most two variables. 

The support for our conjecture (that Problem~\ref{pm:hugo} has a positive answer) is based on the fact that Makanin's algorithm, which decides whether $U=V$ is solvable, can be modified to decide whether there are infinitely many solution. This has been shown by Plandowski and Schubert \cite{PlandowskiS2018tcs}. 
The reduction is simple, so we briefly explain it.

Consider any finite system of word equations~$\cS$.
For satisfiability, that is deciding the existence of a \solu, it is enough to consider a shortest \solu. 
If its exponent of periodicity is sufficiently large, then, by  \cite{hme76}, a shortest \solu generates for every $n\in \N$ another \solu
having an exponent of periodicity larger than~$n$.  

However, for deciding whether~$\cS$ has infinitely many \solu, it does not suffice to consider just shortest \solu{s}.
Let $n=\Abs{\cS}$ be the input size of the system.\footnote{Say, $\Abs{\cS}=1 +|\Gam| + |\cX|+\sum_{(U=V)\in \cS} |UV|$.}
We need an effective upper bound (in terms of~$n$) on the maximal exponent of periodicity $\exp(\sigma)$ for a solution $\sigma$ minimizing the total length $\sum_{(U=V) \in \cS}|\sigma(UV)|$. Given such a computable upper bound $n\mapsto f(n)$, we can decide whether there are infinitely many solutions as follows: 
It is known (see \eg \cite[p.~432]{die98lothaire}) that Makanin's algorithm produces for every $N$, a finite directed \emph{search graph} $\mathcal{M}(\cS, N)$ describing all solutions in the set
\begin{equation}\label{eq:expUVN}
  \mathrm{Sol}(\cS, N) = \{ \sigma \in \Sol(\cS) \mid \exp(\sigma) \leq f(N) \}.
\end{equation}
Moreover, $\mathcal{M}(\cS, f(\Abs{\cS}))$ has the property that if it contains a nontrivial strongly connected component, then  $|\mathrm{Sol}(\cS)| = \infty$. 
Now, for any upper bound $N$ on $f(\Abs{\cS}))$ it holds that as soon as there is a \solu $\sig$ such that 
$\exp(\sig) \geq N$, then $|\mathrm{Sol}(\cS)| = \infty$. 
To reduce the problem, 
we guess a constant~$a\in \Gam$ and a variable $X\in \cX$.
Then, following Plandowski and Schubert, we introduce three fresh variables $P,Z,Q$ and add an additional equation
$X = P(aZ)^{N}Q$. 
This yields a system $\cS_N$ of word equations, and we can decide whether there is a \solu using  Makanin's algorithm. 
If $\mathrm{Sol}(\cS_N)\neq \es$, then $|\mathrm{Sol}(\cS)| = \infty$.

If $\cS_N$ has no \solu, then we construct the search graph $\mathcal{M}(\cS, N)$ and check whether it contains a nontrivial strongly connected component.
If so, then $|\mathrm{Sol}(\cS)| = \infty$ and if not, then $|\mathrm{Sol}(\cS)| < \infty$.
This gives an effective decision procedure to decide whether $\mathrm{Sol}(\cS)$ has infinitely many \solus.\footnote{The paper 
\cite{PlandowskiS2018tcs} yields a polynomial-time reduction of the finiteness problem for $\Sol(\cS)$ to the satisfiability problem for a finite system of word equations.
Firstly, they use that $N$ can be chosen to be singly exponential in the input size and, secondly, they use fast exponentiation via additional variables.}
The remaining obstacle to resolving \prref{pm:hugo} is that we cannot yet exclude that $\cS$ has infinitely many solutions $\sigma$, while $\exp(\sigma)$ remains bounded.\footnote{There currently seem to be no ideas on how to construct such an example, and another conjecture suggests that AI-based computer searches will not be helpful in finding one.}   

In this paper, we focus on equations over torsion-free groups. Makanin showed in \cite{mak83, mak84} that both, the existential theory of equations and the positive theory in free groups are decidable. In this setting, let $F=F(\Gamma_+)$ be a free group over a finite basis $\Gamma_+ \subseteq F$. Thus $\Gamma = \Gamma_+ \cup \Gamma_+^{-1}$ is a set with an involution $a \mapsto \bar a = \oi a$.
Likewise, the set of variables~$\mathcal{X}$ comes with an involution $X \mapsto \bar{X}$. Both involutions are without fixed points, and we only allow mappings $\sigma : \mathcal{X} \to F$ that respect the involution: we require $\sigma(\bar{X}) = \overline{\sigma(X)}$ for all variables. 

To solve an equation $W=1$ in a free group, the immediate idea is to use the fact that there is a canonical bijection between the group and the regular set of reduced words $\mathrm{NF} \subseteq \Gamma^*$ (words without any factor $a\bar{a}$ for $a \in \Gamma$). We represent an element in $F$ by its unique reduced normal form in $\mathrm{NF}$. More precisely, for $g \in F$, we let $\mathrm{nf}(g) \in \mathrm{NF}$ be the unique word such that $\pi(\mathrm{nf}(g)) = g$, where $\pi : \Gamma^* \to F$ is the canonical epimorphism.

This enables a reduction to the existential theory of word equations $U=V$ in the free monoid with involution $(\Gamma \cup \mathcal{X})^*$, under the constraint that a solution $\sigma$ maps variables to reduced words. Crucially, in Makanin's proof, the result regarding the exponent of periodicity can be extended to free monoids with involution. If $\sigma$ is a shortest solution of an equation, then $\exp(\sigma)$ can again be bounded by a singly exponential function. 

Remarkably, the result also holds in the opposite direction: if a system $\cS$ of equations has infinitely many solutions $\sigma : \mathcal{X} \to F$, then $\exp(\cS) = \infty$ where 
\[
  \exp(\cS) = \sup \{ \exp(\mathrm{nf}(\sigma(X))) \mid X \in \mathcal{X} \wedge \sigma \text{ solves } \cS \text{ in } F \}.
\]

According to personal communication with the authors, this positive result was established during the work on the positive solution of Tarski’s conjectures by Kharlampovich and Miasnikov~\cite{KMIV06}, as well as Sela~\cite{sela13}. 
Thus, the answer to Problem~\ref{pm:hugo} is `yes' for free groups;\footnote{This may be viewed as further support for our conjecture in the free monoid case.} however, an explicit statement of this specific form was not provided in their original papers. 

There is another (much simpler) class of groups where we know that solution sets for systems with infinitely many solutions have an infinite exponent of periodicity (with respect to their usual normal forms): the class of finitely generated free Abelian groups.
This raises the natural question of what whether this is also true for classes `in between'.

The most well-studied class sitting between free groups and free Abelian groups consists of (torsion-free) free partially commutative groups, which are also called right-angled Artin groups. 
It is known from~\cite{dm06} that the existential theory of these groups is decidable.
This is achieved by a reduction to the existential theory of free partially commutative monoids with involution.
The decidability of the existential theory of these monoids, without involution but with recognizable constraints, is due to Matiyasevich~\cite{mat97lfcs}; see also \cite{dmm99tcs}.
To handle involutions, it was necessary to find a normal form compatible with the involution. 
No such normal form existed in the literature before~\cite{dm06}.\footnote{The normal form's name was not Susan; so the authors of~\cite{dm06}  were not \emph{desperately seeking} it.}

For the purpose of the introduction, let $\Gamma$ be a finite set with an involution $a \mapsto \bar{a}$; that is, with $\bar{\bar{a}} = a$ for all $a \in \Gamma$.
This involution is extended to $\Gamma^*$ by stipulating $\overline{uv} = \bar{v}\bar{u}$ for all $u, v \in \Gamma^*$.
Moreover, let $I \subseteq \Gamma \times \Gamma$ be an irreflexive symmetric relation with $(a, b) \in I \iff (b, \bar{a}) \in I$. 
We then define the \emph{free partially commutative monoid} $M(\Gamma, I)$ with involution and the  corresponding \emph{free partially commutative group} $G(\Gamma, I)$ as 
\begin{align} \label{eq:MGamI}
  M(\Gamma, I) = \Gamma^* / \{ ab = ba \mid (a, b) \in I \}&
  \text{ and } 
  G(\Gamma, I) = M(\Gamma, I) / \{ a\bar{a} = 1 \mid a \in \Gamma \}.
\end{align}

Since $\Gamma$ is finite, we only consider finitely generated free partially commutative monoids and groups here. 
Free partially commutative monoids (without an involution) were introduced by Cartier and Foata~\cite{cf69}. 
Following computer science notation, an element in $M(\Gamma, I)$ is called a \emph{trace}, and $M(\Gamma, I)$ is also known as a \emph{trace monoid}.\footnote{The term is due to Mazurkiewicz \cite{maz77}; it should not be confused with the trace of a matrix.}

A free partially commutative group (in the above sense) is torsion-free \IFF the involution on $\Gam$ is without fixed points, that is, $\bar{a} \neq a$ for all $a \in \Gamma$. 
These groups were first studied by Droms~\cite{dro85}, who called them \emph{graph groups}. In a more geometric interpretation, a graph group is equivalent to a \emph{right-angled Artin group} (RAAG).
Note that, in general, \eqref{eq:MGamI} also includes every \emph{right-angled Coxeter group} (RACG), where the free partially commutative group $\GGI$ is such a group if every letter in $\Gam$ is self-involuting, that is, $\bar a = a$ for all $a \in \Gam$.

Let us say a little bit more about the the complexity of deciding the existential theories
in the algebraic structures defined by \eqref{eq:MGamI}.
The first estimates of the complexity of Makanin's algorithm \cite{mak77} led to upper time bounds involving a tower of six or seven exponentials. 
Using known facts about solving systems of linear equations over the rationals, it was not too hard to find an exponential space bound \cite[Thm.~12.4.2]{die98lothaire}.
A different approach was taken by Plandowski and Rytter \cite{pr98icalp} who applied Lempel-Ziv encodings to show that the satisfiability of word equations in free monoids is decidable in \textsc{\small PSPACE}. 
They also conjectured that the satisfiability problem for word equations is \textsc{\small NP}-complete, which motivated a revival of the research on Diophantine problems in other algebraic structures. 
For example, it was possible to prove a \textsc{\small PSPACE} upper bound for the existential theories in free partially commutative monoids and groups, which were then subsumed by the results in \cite{dl08ijac} on the existential theory in graph products. 

\begin{conj}\label{conj:PR98}
  The satisfiability of word equations in free monoids is \textsc{\small NP}-complete. 
\end{conj}

Kharlampovich et al.~\cite{KharlampovichLMT10} settled that \textsc{\small NP}-completeness holds for 
the satisfiability of quadratic equations over free groups. 
But so far, the attempts to prove the analogue for quadratic equations over free monoids failed, and \prref{conj:PR98} remains wide open.

Shortly after, Je\.z~\cite{jez13stacs} came up with the intriguing idea of \emph{recompression}:
his conference paper presented in less that 12~pages an algorithm (and a correctness proof thereof) which showed that the satisfiability of word equations in free monoids is in \textsc{\small NSPACE}($n \log n$).\footnote{Without his expectation estimates, 
the proof is shorter and still gives an \textsc{\scriptsize NSPACE}($n^2 \log n$) bound.} This has since been improved to a non-deterministic linear space bound by Je\.z~\cite{jez22}.\footnote{Je\.z also showed that the set of all satisfiable word equations is context-sensitive.}

Analyzing the recompression technique in more detail revealed another unexpected result on the structure of the full solution set of a word equations in free groups, hyperbolic groups, and free partially commutative monoids and groups. 
In each case the full solution set is an \textsc{\small EDT0L}-language, which is perhaps the most interesting class among languages given by Lindenmayer systems.
We refer to \cite{RozS86} for a collection of results about Lindenmayer systems, and for the characterization of 
\textsc{\small EDT0L} by Asveld \cite{Asveld1977}.
In \cite{CiobanuDiekertElder2016ijac} it was shown that the full solution set in free groups is \textsc{\small EDT0L}. 
This was generalized to RAAGs~\cite{DiekertJK2016}, to virtually free groups~\cite{DiekertElder2020ijac}, and to hyperbolic groups~\cite{CiobanuElder2021}. 
More recently, Levine~\cite{Levine2023} considered group extensions, and Duncan, Evetts, Holt, and Rees~\cite{DuncanEvettsHoltRees2023} used \textsc{\small EDT0L}-systems to study equation in solvable Baumslag-Solitar groups. 
The story is bound to continue.\footnote{However, brief histories have to stop; even \emph{The Neverending Story}~\cite{Ende2009} ends after 448~pages.}
 
\subsection{Our Main Results}\label{sec:mainres}

Our results concern the family $\cG$ consisting of triples $[G, \pi, \nf]$ where $G$ is a group, $\pi: \Gam^* \to G$ an epimorphism, and $\nf: G \to \Gam^* $ a \emph{normal form} such that:\footnote{As usual, objects in the family $\cG$ are abstract and defined up to isomorphism.}
\begin{enumerate}
\item The group $G$ is torsion free and the set $\sqrt[2]{g} = \{ h \in G \mid h^2 = g \}$ is finite  for every $g \in G$.
\item The domain of $\pi: \Gamma^* \to G$ is a finitely generated free monoid~--~hence, $G$ is finitely generated~--~and $\mathrm{nf}: G \to \Gamma^*$ is a section of $\pi$, meaning that $\pi(\nf(g)) = g$ for all $g \in G$.
\item For all $u, p, v \in G$ with $p \neq 1$ and for every $n \in \mathbb{N}$, there exists an $N \in \mathbb{N}$ such that the word $\mathrm{nf}(up^N v)$ contains a nonempty factor $q^n \in \Gamma^+$. We denote this property (which depends on the normal form $\mathrm{nf}$) as $\exp(\mathrm{nf}(up^* v)) = \infty$.
\end{enumerate}

Membership in $\mathcal{G}$ crucially depends on the choice of the pair $(\pi, \mathrm{nf})$. 
If a finitely generated group $G$ satisfies the first item, then we might artificially define $\pi : \Gam^* \to G$ and $\nf: G \to \Gam^*$ in such a way that $[G, \pi, \mathrm{nf}] \notin \mathcal{G}$.
In our examples, we mostly investigate groups $G$ alongside some `natural' choice of the pair $(\pi, \mathrm{nf})$.
For instance, $[\mathbb{Z}, \pi, \nf]$ belongs to $\mathcal{G}$ for the most obvious choice of $(\pi, \nf)$, viz.\ the epimorphism $\pi: \Gam^* \to \Z$ induced by the inclusion $\Gam = \os{\pm 1} \sse \Z$ and its unique section $\nf: \Z \to \Gam^*$ with image $\nf(\Z) = \os{-1}^* \cup \os{+1}^*$.

Our first main result on the family $\cG$ is the following closure property.
\begin{thm}\label{thm:main1}
Let $\mathbb{G}$ be a graph product over finitely many groups $G$ with $[G, \pi_G, \mathrm{nf}_G] \in \mathcal{G}$. Then $[\mathbb{G}, \pi_{\mathbb{G}}, \mathrm{nf}_{\mathbb{G}}] \in \mathcal{G}$ for a natural construction of $(\pi_{\mathbb{G}}, \mathrm{nf}_{\mathbb{G}})$ depending only on the $(\pi_G, \mathrm{nf}_G)$.
\end{thm}

Our second main result, and our primary motivation for studying the family $\cG$, concerns the exponent of periodicity for solution sets of quadratic equations.
\begin{thm}\label{thm:main12}
Let $[G, \pi, \mathrm{nf}] \in \mathcal{G}$.
Further, let $\mathcal{S}$ be a finite quadratic system of equations over the group $G$ in a set of variables $\mathcal{X}$. 
If the system $\mathcal{S}$ has infinitely many solutions, then there exists a variable $X \in \cX$ such that $\exp(\mathrm{nf}(\{ \sigma(X) \mid \sigma \textup{ solves } \mathcal{S} \})) = \infty$. 
\end{thm}

Combining the above results, we immediately obtain the following.

\begin{cor}\label{cor:main123}
  Let $\mathbb{G}$ be a graph product over finitely many groups $G$ with $[G, \pi_G, \mathrm{nf}_G] \in \mathcal{G}$, and let $(\pi_\G, \nf_\G)$ be as in \prref{thm:main1}.
  If a finite quadratic system of equations $\mathcal{S}$ over $\mathbb{G}$ has infinitely many solutions, then $\exp(\mathrm{nf}_{\mathbb{G}}(\{ \sigma(X) \mid \sigma \textup{ solves } \mathcal{S} \})) = \infty$ for some variable $X$. 
\end{cor}

As an abstract result, \prref{cor:main123} is most interesting when applied to natural triples. 
For example, finitely generated free groups belong to $\cG$ with the standard presentation and normal forms given by reduced words.
This was shown first with a combinatorial proof by Bastien Laboureix (unpublished) during an internship 
at Stuttgart in 2019.\footnote{His proof 
followed an outline communicated by Olga Kharlampovich.}

On the other hand, \prref{cor:main123} is more powerful. 
As shown above, we have $[\Z,\pi,\nf]\in \cG$  with the most natural choices. 
Hence, the example of free groups is just a special case of the following corollary for right-angled Artin groups.

\begin{cor}\label{cor:raagini}
Suppose that $G = G(\Gamma, I)$ is a right-angled Artin group.
Then $[G, \pi, \mathrm{nf}_{\mathrm{slex}}] \in \mathcal{G}$ where $\pi: \Gam^* \to G$ is the canonical epimorphism and $\mathrm{nf}_{\mathrm{slex}}: G \to \Gam^*$ is the short-lex normal form with respect to any linear order $\leq$ on $\Gam$.
Thus, in particular, Theorem~\ref{thm:main12} applies.
\end{cor}

Moreover, the family $\cG$ is much larger than the family of right-angled Artin groups:
we show that it includes graph products over finitely many
groups from the following list.

\begin{itemize}
  \item An infinite class of \BSG{s}, which we characterize in \prref{sec:BS}.
  \item An infinite class of (strongly) polycyclic groups properly containing all finitely generated torsion-free nilpotent groups; see \prref{sec:WIP}. 
  \item The class of all torsion-free hyperbolic groups; see \prref{sec:hype}.
\end{itemize}

On our rather long way to show the main results, we present some original 
proofs.\footnote{First ideas for the paper came to light, soon after the joint Miasnikov- and GAGTA-conference at Stevens Institute (Hoboken, NJ) while flying during sunrise in June 2025 across Tipperary (Ireland) on the long way to go home. It was indeed a long, long way to go 
before a draft of this paper was ready.} 
For example, we show that a any graph product of torsion-free monoids is torsion free (\prref{thm:gptor}), which generalizes a result of Green for graph product of groups
\cite[Thm.~3.26]{green1990graph}. 

\section{Preliminaries}\label{sec:prel}

We use standard notation of set theory. 
Frequently, we identify an element with the singleton containing this element. 
The restriction of a mapping $\psi: B\to C$ to some $A \sse B$ is still denoted $\psi$.
The set of mappings from $A$ to $B$ is denoted by $B^A$.
If $\phi: A\to B$ and $\psi: B\to C$ are mappings then the composition $\psi \phi:A\to C$ is defined by $\psi \phi(a)= \psi(\phi(a))$. 

By $\N$ and $\Z$ we denote the set of natural numbers and integers, respectively;
the set~$\N_+$ consists of the positive natural numbers.
For monoids $N,M$ we write $N\leq M$ if $N$ is a submonoid of $M$, and if $M=G$ is a group, then 
$N\trianglelefteq G$ means that $N$ is a normal subgroup.

Let $\Sig$ be a set, then $\Sig^*$ denotes the free monoid over $\Sig$.
In this context $\Sig$ is a (possibly infinite) \emph{alphabet}, the elements of $\Sig$ are called \emph{letters}, and the elements of $\Sig^*$ are called \emph{words}.
The empty word is denoted by $1$, as is the neutral element in other monoids.
If $w \in \Sig^*$ is a word, then $\abs{w}$ denotes its length and $\abs{w}_a$ denotes the number of occurrences of the letter $a \in \Sig$ within $w$ so that, in particular, $\abs{w} = \sum_{a \in \Sig} \abs{w}_a$. 
If $\Gam$ is finite and $\leq$ is a linear order on $\Gam$, then every nonempty subset $L\sse \Gam^*$ has a \emph{short-lex} first word, which is lexicographically first with respect to $\leq$ among the words of minimal length in $L$.  

A monoid is \emph{finitely generated}, abbreviated \emph{\fg}, if there is an \epi 
of monoids $\pi:\Gam^*\to M$ where $\Gam$ is a finite set, and 
therefore, $\pi(\Gam)\sse M$ is finite generating subset of $M$.

An element~$x$ of a monoid~$M$ with $x \neq 1$ is called a \emph{torsion element} if there are $r,p\in \N$ with $p\geq 1$ such that $x^{r+p}=x^r$. 
We call~$M$ \emph{torsion free} if it contains no such element.

For~$x \in M$, we also define its set of \emph{square roots} as~$\sqrt[2]{x} = \set{y \in M}{y^2 = x}$ and we say that~$M$ has the \emph{\fsqrtp} if~$\sqrt[2]{x}$ is finite for every element~$x \in M$.

\subsection{Regular, Rational, and Recognizable Sets}\label{sec:regrecrat}

The family of regular sets over a finite alphabet $\Gam$, which we denote by $\Reg(\Gam^*)$, is defined as usual in formal language theory, see \eg~\cite{eil74,HU} or any other textbook about formal languages.\footnote{A subset in a \fg~free monoids is also called a language.} 
For any monoid $M$, we denote the family of \emph{rational subsets} of~$M$ by $\Rat(M)$. 
It consists of those subsets $L \sse M$ which can be written as the image $L = \psi(K)$ of some regular set $K \in \Reg(\Gam^*)$, where $\Gam$ is a finite alphabet, under a homomorphism $\psi: \Gam^* \to M$.

There are at least two other common characterizations of
$\Rat(M)$. The first one is by \emph{regular expressions.}
It say that $\Rat(M)$ is the least family of 
subsets which contains all finite subsets of $M$ and which is closed under finite union, concatenation, and the \emph{star-operator} which is defined for $L\sse M$ 
by the union $L^*=\bigcup\set{L^i}{i\in \N}$ 
where $L^0=\os 1$ and $L^{i+1}=L^{i}\cdot L$ for $i\in \N$.
That is, $L^*$ is the submonoid of $M$ which is generated by $L$.

The second characterization is based on the concept of nondeterministic automata. 
An \emph{$M$-automaton} $\cA$ 
is a tuple $\cA=(Q,M,\del,I,F)$ where $Q$ is the set of \emph{states}, $I\sse Q$ is the set of \emph{initial states}, 
$F\sse Q$ is the set of \emph{final states}, and $\del\sse Q \times M \times Q$ is the \emph{\tra relation}. 
The \emph{accepted language} of the $M$-automaton $\cA$ is defined as usual: 
\[
  L(\cA)=\set{x\in M}{\text{$x$ is the label of some path from an initial state to a final state}}. 
\]

The automaton $\cA$ is called \emph{complete} if 
for every $p\in Q$ and $x\in M$ there is a path from $p$ to some state $q$ which is labeled by $x$.
It is called \emph{deterministic} if there is 
exactly one initial state $q_0$ and for each state $p$ there is at most one~$q$ such that there is a path labelled by~$x$ from~$p$ to~$q$.
In this case we also write $q=p\cdot x$. If $\cA$ is complete and deterministic, then $(p,x)\mapsto p\cdot x$ defines a right-action of $M$ on~$Q$.
Finally, if $Q$ and $\del$ are finite, then we say that 
$\cA$ is \emph{finite}. 
A nondeterministic finite $M$-automaton 
is abbreviated as an $M$-NFA, and it is abbreviated as an $M$-DFA if $\cA$ is deterministic.  
It is easy to see that $L\in \Rat(M)$ \IFF $L$ is accepted by some $M$-NFA $\cA$, meaning that $L = L(\cA)$.

\medskip

Apart from rational subsets, there is also the family $\Rec(M)$ of \emph{recognizable} subsets of~$M$.
It consists of those subsets $L \sse M$ that are \emph{recognized} by some \hom $\mu : M \to N$ to a finite monoid $N$, which means that $L = \oi \mu(\mu(L))$. 

For finitely generated free monoids, the above concepts coincide: Kleene's Theorem~\cite{kle56} asserts that, for every \fg~free monoid $\Gam^*$, every $L\in \Rat(\Gam^*)$ is accepted by some complete DFA~$\cA$ and is thus recognized by the induced homomorphism to the transformation monoid on the set of states of $\cA$.
As a consequence $\Reg(\Gam^*)=\Rat(\Gam^*)=\Rec(\Gam^*)$.

On the other hand, if $M=G$ is an infinite group, then its nonempty finite subsets are rational, but they are never recognizable because the kernel $\oi \mu(1)$ of any \hom $\mu:G\to N$ to a finite monoid $N$ is a subgroup of finite index in $G$ and, hence, infinite.

\medskip

For the following facts on rational and recognizable sets, we refer to \cite{eil74} and \cite[Sec.~7]{edam16}.
\begin{itemize}
\item A monoid~$M$ is \fg \IFF $\Rec(M) \sse \Rat(M)$ \IFF  $M\in \Rat(M)$.
\item For every \hom $\phi: M \to M'$, the following implications hold.
  \begin{itemize}
    \item If $L \in \Rat(M)$, then $\phi(L) \in \Rat(M')$.
    \item If $L' \in \Rec(M')$, then $\oi\phi(L') \in \Rec(M)$.
  \end{itemize}
\item The family $\Rec(M)$ is a Boolean algebra with respect to the standard set operations. 
\item If $M$ contains the free product $\N_+ \ast (\N_+ \times \N_+)$, then $\Rat(M)$ is not closed under finite intersection.\footnote{In fact, closure of $\Rat(M)$ under finite intersection should be viewed as a rare exception.}
\Ip this implies $\Rat(M) \neq \Rec(M)$.
\end{itemize}

If we speak about a regular set, then we always refer to a subset of a \fg~free monoid.
This is to avoid confusion, since some literature defines $\Reg(M)$ for a general monoid $M$ as either $\Rat(M)$ or as $\Rec(M)$ (or sometimes even as something completely different). 

\section{The Exponent of Periodicity}\label{sec:expo}

Throughout this section, let $\Gam$ be a finite alphabet.
Then the \emph{exponent of periodicity} of a word $w \in \Gam^*$ and of a language $L \sse \Gam^*$ are defined as follows.
\begin{align}\label{eq:wexp}
\exp(w)= \max\set{e\in \N}{w=u\, p^{e}\, v \wedge p\neq 1} \in \N
\end{align}
\begin{align}\label{eq:expL}
\exp(L)= \sup\set{\exp(w)}{w\in L}\in \N\cup \os\infty.
\end{align}
We are mainly interested in the asymptotical behavior of the exponent of periodicity.
To this end, we call a language~$L \subseteq \Gam^*$ \emph{periodically nice} if the equivalence~$\exp(L) = \infty \Leftrightarrow |L| = \infty$ holds; thus a periodically nice language is either finite or has infinite exponent of periodicity.

Sometimes we also deal with \emph{periodically perfect} languages~$L \subseteq \Gam^*$. These are those languages which satisfy a much stronger property:
for every~$n \in \N$ there exists an~$N \in \N$ such that~$|w| \geq N$ implies~$\exp(w) \geq n$ for all~$w \in L$.
Informally, a language $L$ is periodically perfect if every long word $w\in L$ has high exponent of periodicity. 
\prref{thm:perper} below shows that, for regular languages, these two notions are closely related.

To extend the concept of the exponent of periodicity to arbitrary finitely generated monoids, we make use of normal forms.

\begin{defi}\label{def:nf}
Let $\pi:\Gam^*\to M$ be an \epi onto a monoid $M$.
A \emph{normal form} (with respect to $\pi$) is a mapping $\nf : M \to \Gam^*$ such that $\pi(\nf(x)) = x$ for all $x \in M$.
We say that $\nf$ is \emph{geodesic} if the normal form $\nf(x)$ of every $x \in M$ has minimal length in $\oi\pi(x)$.
\end{defi}

The two specific types of normal forms that we are interested in
are defined next. We allow that $\nf{1}\neq 1\in \Gam^*$ because the 
the proofs do not need $\nf(1)=1$. On the other hand, in all concrete
constructions and examples $\nf(w)$ does not have a nontrivial 
factor $u$ with $\pi(u)=1$. This is also true for the `weird' normal forms in \prref{ex:fibnfs}.

\begin{defi}\label{def:expinfty}
  Let $\nf:M\to \Gam^*$ be a normal form with respect to $\pi: M \to \Gam^*$. 
\begin{itemize}
  \item The normal form $\nf$, and the pair $(\pi, \nf)$, are called \emph{admissible} if the set $\nf(up^*v)$ is periodically nice for all $u,p,v \in M$. That is, $\exp(\nf(up^*v))=\infty$ whenever $up^*v$ is infinite.
  \item The normal form $\nf$, and the pair $(\pi, \nf)$, are called \emph{perfectly admissible} if $\nf(up^*v)$ is periodically perfect for all $u,p,v \in M$. That is, if for for each $n\in \N$ there is some $N\in \N$ such that every word $w\in \nf(up^*v)$ with $\abs{w} \geq N$ satisfies $\exp(w)\geq n$.
\end{itemize}
\end{defi}

Note that a perfectly admissible pair $(\pi,\nf)$ is admissible (since we assume throughout that $\Gam$ is finite). 
Note also that using normal forms is essential for defining the exponent of periodicity for elements and sets of monoids.
For example, if a monoid $M$ contains some subgroup $G$ which has an element $g$
of infinite order, then the set $\set{g^ng^{-n}}{g\in G \wedge n\in \N}$ would have an infinite exponent of periodicity under the na\"ive definition, even though it has but a single element.
However, when using normal forms, the property whether $\exp(\nf(S))=\infty$ for some $S\sse M$ depends highly on the choice of the normal form. 
The following example shows that it is quite easy to construct `weird' normal forms which are geodesic but still not admissible.
The example uses the additive monoid $\N$, but it clearly extends to every 
monoid or group with an infinite monogenetic submonoid.\footnote{Recall that 
a monogenetic submonoid refers to a submonoid generated by a single element.} 
\begin{exa}[Weird Normal Forms]\label{ex:fibnfs}
Consider a two-letter alphabet~$\Gam = \os{a, b}$ and the presentation~$\pi: \Gam^* \to \mathbb{N}$ given by~$a \mapsto 1$ and~$b \mapsto 1$.
In order to construct a geodesic normal form~$\nf:\N\to \os{a, b}^*$, we simply let $\nf(n)$ be the prefix of length $n$ of the (infinite) Thue-Morse word $t = abbabaabbaababba\cdots$, which is the fixed point of the homomorphism defined by $a \mapsto ab, b \mapsto ba$.
Since the Thue-Morse word $t$ is cube free \cite[Prop.~3.1.1]{lot02}, we have $\exp(\nf(n)) \leq 2$ for all $n \in \mathbb{N}$.
\hspace*{\fill}$\diamond$
\end{exa}

On the other hand, the \emph{usual} normal forms of the free group~$F(\Gam)$, of the free Abelian group~$\Z^n$ and its submonoid $\N^n$, as well as of any finite monoid are all perfectly admissible.

\begin{rem}\label{rem:wplexnf}
We are mainly interested in \fg~monoids with a decidable word problem. 
Suppose that $\pi:\Gam^*\to M$ is an \epi where $\Gam$ is finite
and $\leq$ is a linear order on $\Gam$, and that $\nf:M\to \Gam^*$ is any normal form where $\nf(M) \sse \Gam^*$ is a recursively enumerable.
Let us also consider the short-lex normal form $\nfslex : M \to \Gam^*$, where $\nfslex(x)\in \Gam^*$ is the lexicographically first among all shortest words in $\oi\pi(x)$.
Then the following are equivalent.
\begin{enumerate}
\item The monoid $M$ has a decidable word problem.
\item On input $w\in \Gam^*$ we can compute 
the the short-lex normal 
form $\nfslex(x)$. 
\item On input $w\in \Gam^*$ we can compute the normal form $\nf(\pi(w))$. 
\end{enumerate}

The implication 
$(1) \Rightarrow (3)$ is standard: we start $w\in \Gam^*$ 
and the run the enumeration to list all words in $\nf(M)$, we stop at the first hit where the enumeration outputs a word $u$ such that 
$\pi(u)=\pi(w)$. This implies $\nf(w)=u$.
For the other direction $(3) \Rightarrow (1)$ we compute
on $u,v\in \Gam^*$ the normal forms $\nf(u)$ and $\nf(v)$.
Clearly, $\pi(u)=\pi(v)\iff \nf(u)=\nf(v)$. 
We are done because short-lex normal 
forms are computable. 
\hspace*{\fill}$\diamond$
\end{rem}

\begin{lem} \label{lem:nice_perfect}
  For every language $L \subseteq \Gam^*$, the following assertions are equivalent.
  \begin{enumerate}
    \item The language $L$ is periodically perfect.
    \item Every subset $K \subseteq L$ is periodically perfect.
    \item Every subset $K \subseteq L$ is periodically nice.
  \end{enumerate}
  In particular, any periodically perfect language is periodically nice.
\end{lem}
\begin{proof}
  The equivalence~$(1) \Leftrightarrow (2)$ and the implication~$(2) \Rightarrow (3)$ are straightforward. 
  The proof of~$(3) \Rightarrow (1)$ is by contraposition. 
  If~$L$ is not periodically perfect, then there exists some~$n \in \N$ such that~$L$ contains arbitrarily long words with exponent of periodicity less than~$n$.
  These words form a subset~$K \subseteq L$ which is not periodically nice.
\end{proof}

All regular and context-free languages are periodically nice
by the corresponding pumping lemmas.
On the other hand, $\Gam^*$ is not periodically perfect, provided that $\Gam$ has at least two elements.
In fact, being periodically perfect is, in some sense, quite rare even for regular languages, as is evident from the following characterization. 

\begin{thm}\label{thm:perper}
Let $L\sse \Gam^*$ be a regular language and $L=L(\cA)$ where $\cA$ is a deterministic and trim automaton\footnote{Trim means that every state is on some accepting path.} accepting $L$.
Then the following assertions are equivalent. 
\begin{enumerate}
\item The stabilizer of each state is of the form~$w^*$ for some~$w \in \Gam^*$. That is, there is at most one simple cycle attached to every state of~$\cA$. 
\item The regular language $L$ is periodically perfect. That is, for every~$n \in \N$ there exists an~$N \in \N$ such that~$|w| \geq N$ implies~$\exp(w) \geq n$ for all~$w \in L$.
\end{enumerate}
\end{thm}

\begin{proof}
The assertion holds if $L(\cA)$ is finite. 
Therefore, we may assume that $L(\cA)$ is infinite.

Suppose that the stabilizer of each state is of the form~$w^*$ for some~$w \in \Gam^*$. Let $n\in \N$. 
Then the accepting path of every sufficiently long word in $x \in L(\cA)$ has to visit some state at least $n$ times. 
Since this state has a stabilizer of the form $w^*$ for some $w \in \Gam^*$ with $w \neq 1$, we can factorize the word $x$ as $x=uw^nv$. 
This shows that $L(\cA)$ is periodically perfect.

For the converse, assume that there is a state $q$ stabilized by both $u^*$ and $v^*$ for some words $u,v \in \Gam^*$ with $u\notin v^*$ and $v\notin u^*$. By the Defect Theorem (see \eg \cite[Thm.~6.2.1]{lot02}) either $\os{u,v}$ is a code with $u\neq v$  or there is some $w\in \Gam^*$ such that $u= w^k$ and $v=w^\ell$. 

In the second case, we may assume without restriction that $k<\ell$.
Since $u\notin v^*$, $v\notin u^*$, and
$\cA$ is deterministic, there is also a loop around~$q$ labeled by $v' = w^{\ell-k}$ and the loop around~$q$ labeled by $v$ was not simple. 
Continuing in this way, we eventually see that either the stabilizer of $q$ is indeed of the form~$w^*$ for some word $w \in \Gam^*$, or we reach a situation in which $\{u, v\}$ is a code with $u \neq v$. 
We claim that the latter cannot occur.

Indeed, since $\cA$ is trim, there exists $r,s \in \Gam^*$ such that $q_0 \cdot r = q$, where $q_0$ is the starting state of $\cA$, and such that $q \cdot s$ is a final state.
Then the set $r \{u,v\}^* s$ belongs to $L = L(\cA)$, but contains arbitrarily long words with bounded exponent of periodicity whenever $\os{u, v}$ is a code with $u \neq v$.
For example, we can consider consider the set $T \sse \os{a, b}^*$ of all finite prefixes of the Thue-Morse word~$t$ as in \prref{ex:fibnfs}.
Upon applying the injective homomorphism $\varphi: \os{a,b}^* \to \Gam^*$ defined by $\varphi(a) = u$ and $\varphi(b) = v$, we obtain the set $\varphi(T) \sse \os{u,v}^*$.
By the following \prref{lem:codes}, this set has bounded exponent of periodicity.
Hence, $K = r \varphi(T) s$ is an infinite subset of $L$ with bounded exponent of periodicity, contradicting \prref{lem:nice_perfect}.
\end{proof}

To the best of our knowledge, there are so far no published proofs of the statement in the following lemma. 
\begin{lem}\label{lem:codes}
  Let $h : \Gam^* \to \Del^*$ be an injective homomorphism for some finite alphabet $\Gam$.
  Then it holds that $\exp(h(w)) \in \Theta(\exp(w))$ for all $w \in \Gam^*$ as $\exp(w)$ tends to infinity.

  In particular, every $L \sse \Gam^*$ satisfies $\exp(L) = \infty$ \IFF $\exp(h(L)) = \infty$.
\end{lem}

\begin{proof}
  Since the homomorphism $h$ maps nonempty words to nonempty words, we clearly have $\exp(w) \leq \exp(h(w))$ for all $w \in \Gam^\ast$.
  Conversely, let $w \in \Gam^*$ with $k = \exp(h(w))$ sufficiently large.
  More specifically, we require that $k \geq 2m^2$ where $m = \max\set{\abs{h(a)}}{a \in \Gam}$.

  By assumption, we have $h(w) = u p^k v$ for some $u,v \in \Del^*$ and $p \in \Del^+$.
  For every $0 \leq i \leq m$, let us write $w_i$ for the minimal prefix of $w$ such that $u p^{i \cdot m}$ is a prefix of $h(w_i)$.
  Then the difference in length $d_i = \abs{h(w_i)} - \abs{u p^{i\cdot m}}$ is bounded by $0 \leq d_i < m$ for all $0 \leq i \leq m$.
  Let us now fix $i$ and $j$ with $0 \leq i < j \leq m$ and $d_i = d_j$; such indices exist by the pigeon hole principle.

  We then find ourselves in the following situation: $up^{i \cdot m}$ is a prefix of $h(w_i)$, which in turn is a \emph{proper} prefix of $up^{j \cdot m}$, which is a prefix of $h(w_j)$.
  Moreover, writing $p^m = rs$ with $r,s \in \Del^*$ and $\abs{r} = d_i = d_j$, we have $up^{i \cdot m}r = h(w_i)$ and $up^{j \cdot m}r = h(w_j)$.
  Hence, the unique word $q \in \Gam^+$ with $w_i q = w_j$ satisfies $h(q) = (sr)^{(j-i) \cdot m}$.
  It then follows that $h(w_i q^n) = u p^{i \cdot m} p^{n \cdot (j-i) \cdot m} r$.
  This is a prefix of $h(w) = up^k v$ in case $i \cdot m + n \cdot (j-i) \cdot m + m \leq k$.
  Since the left side of this inequality can be bounded above by $(n+2) \cdot m^2$, the latter holds for $n = \lfloor k / m^2 \rfloor - 2$.

  We have shown that $h(w_iq^n)$ with $n = \lfloor k / m^2 \rfloor - 2$ is a prefix of $h(w)$.
  Since $h$ is an injective homomorphism, it follows that $w_i q^n$ is a prefix of $w$.
  Hence, $\exp(w) \geq n \geq \Omega(\exp(h(w)))$.
\end{proof}

We conclude this section with a variant of the pumping lemma for regular languages.
\begin{lem} \label{lem:pumping_lemma}
    Let~$M$ be a cancellative torsion-free monoid equipped with an admissible normal form~$\nf$.
    Then~$\nf(L)$ is periodically nice for every set~$L \in \Rat(M)$. 
\end{lem}

\begin{proof}
Let~$\cA$ be a finite~$M$-automaton with~$n$ states such that the language~$L = L(\cA)$ is infinite.
We have to show that then~$\exp(\nf(L)) = \infty$ as well.
Since~$L$ is infinite, there must exist an element~$w \in L$ such that the shortest path in $\cA$ accepting~$w$ has length greater than~$n$. It must therefore visit some state of~$\cA$ twice, yielding a factorization
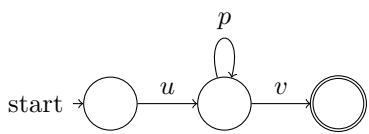
\begin{center}\begin{tikzpicture}[scale=0.3]
  \node[state, initial, scale=0.8] (1) at (0, 0) {};
  \node[state, scale=0.8] (2) at (5, 0) {};
  \node[state, accepting, scale=0.8] (3) at (10, 0) {};

  \draw[->] (1) to node[above] {$u$} (2);
  \draw[->] (2) edge[loop above, looseness=10] node[above] {$p$} (2);
  \draw[->] (2) to node[above] {$v$} (3);
\end{tikzpicture}\end{center}
with~$u, p, v \in M$ and~$p \neq 1$.
Since~$M$ is torsion-free, this implies~$|p^*| = \infty$.
Moreover, cancellativity implies that left and right multiplication are injective in~$M$, and thus~$up^*v$ is infinite as well. Finally, the admissibility of~$\nf$ means that~$\exp(\nf(up^*v)) = \infty$, proving the claim since~$up^*v \subseteq L$ by construction.
\end{proof}

In the following we are interested in examples of the just discussed properties.
We start with the behavior with regard to taking graph products and, afterwards, we examine some important classes of groups and normal forms canonically associated with them.

\section{Free Partially Commutative Monoids and Groups}\label{sec:trtheo}

Our results on graph products rely on the theory of free partially commutative monoids. 

Let $\Gam$ be any set and $I\sse \Gam\times \Gam$ an 
irreflexive and symmetric relation. 
Associated with this data is the \emph{free partially commutative monoid}, as introduced by Cartier and Foata~\cite{cf69},
\begin{align}\label{eq:cf69}
M(\Gam,I)= \Gam^*/\set{(ab=ba}{(a,b)\in I} = \Gam^*/\set{ab = ba}{(a,b)\in I}.
\end{align}
Thus, the monoid $M(\Gam,I)$ comes with a canonical \epi
$\theta: \Gam^* \to M(\Gam,I)$. 
Adopting the computer-science terminology, an element in $M(\Gam,I)$ is called a \emph{trace} and $M(\Gam,I)$ is called a \emph{trace monoid}.\footnote{The term is due to \mazu~\cite{maz77}. It must not be confused with the trace of a matrix.}
Let $u,v$ be words in $\Gam^*$ such that
$\theta(u)$ and $\theta(v)$ are equal as traces in $M(\Gam,I)$, 
then $|u|=|v|$ but, otherwise, the words might have a quite different shape.

The pair $(\Gam,I)$ is also called an \emph{independence alphabet}.\footnote{The wording of independence was coined in computer science, where letters denote events using a nonempty set of resources. Thus, independent events commute, but no event is independent of itself.}
It is viewed as an undirected graph.
Without restriction we assume that $\Gam$ is a set with \invol~--~that is, there is a bijection  $a\mapsto \bar a$ such that $\bar{\bar a}=a$ for all $a\in \Gam$~--~and that the \invol is compatible with the independence relation $I$ in the sense that $(a,b)\in I\iff (\bar a, b)\in I$ for all $a,b \in \Gam$.
(For example, one can always take the identity as an \invol, so that $\bar a = a$ for all $a \in \Gam$.)
The \invol can then be extended to $\Gam^*$ and to $M(\Gam,I)$ in such a way that $\ov{xy}= \bar y  \bar x$ holds for all $x,y$.

Using the \invol, the pair $(\Gam,I)$ defines a canonical quotient group 
of $M(\Gam,I)$ by 
\begin{align}\label{eq:GGamI}
G(\Gam,I)= M(\Gam,I)/\set{a\bar a = 1}{a\in \Gam},
\end{align}
which we call a \emph{free partially commutative group}.
If $a\neq \bar a$ for all $a\in \Gam$, then such a group is also called a \emph{graph group} by Droms, \cite{dro85}, because we can view the pair $(\Gam,I)$ as an undirected graph. The family of graph groups  coincides with the family of \emph{right-angled Artin groups}, or RAAGs.
 If we have $a=\bar a$ for all $a\in \Gam$, then $G(\Gam,I)$ is a \emph{right-angled Coxeter group}.
 
The theory of trace monoids is well-established and the cornerstone to understanding RAAGs and, more generally, free partially commutative groups according to the above definition.
We recall some well-known basic facts from this theory in \prref{sec:bastt}.

\subsection{Rudiments of Trace Theory}\label{sec:bastt}
Using the notation from above, let $u=a_1\cdots a_m\in \Gam^*$ such that $a_i\in \Gam$ for all $1\leq i \leq m$. 
Further, let 
$x=\theta(u)$ be the associated trace in $M(\Gam,I)$.  
In general, the trace $x$ may be represented in this way by many different words $u$.
Nonetheless, $x$ has a well-defined length $|x|= |u|$ and a well-defined $a$-length $|x|_a= |u|_a$ for every $a\in \Gam$.

If the alphabet $\Gam$ is given a linear order $\leq$,\footnote{According to ISO~2382 and DIN~44300, an alphabet is a finite set together with a linear order.} 
then the lexicographical order of $\Gam^*$ defines a normal form $\nflex:M(\Gam,I)\to \Gam^*$ such that $\nflex(x)$ is a shortest word 
in $\oi\theta(x)$ for all $x\in M(\Gam,I)$.
Therefore, in trace monoids we have $\nflex=\nfslex$.

In order to have a unique representation 
of $x$ we define a vertex-labeled directed graph $(\pos(x),E(x),\lam)$ where $\pos(x)$ is the set of vertices, which we call  \emph{positions}, 
and $E(x)$ is the set of directed edges (or \emph{arcs}), and 
$\lam$ is the vertex labeling.
This graph is called the \emph{dependence graph} of $x$. This central notion for trace theory is defined next.\footnote{Some say that Plato wrote about 380 BCE at the entrance of his school: `Whoever is not proficient to deal with dependence graphs should not enter the following sections.'}
\begin{itemize}
\item We let $\pos(x)=\os{1, \ldots, m}$. A vertex in $\pos(x)$ is also 
called a \emph{position}.
\item We let $\lam:\pos(x)\to \Gam$ be the vertex labeling
where $\lam(i)=a_i\in \Gam$.
\item We let $E(x)=\set{(i,j)}{1\leq i < j\leq m\wedge (a_i,a_j)\notin I}$.
\item The \emph{dependence graph} $D(x)$ is the abstract graph
$[V(x),E(x),\lam]$, which is the equivalence class of the concrete graph $(\pos(x),E(x),\lam)$ up to \iso.
\end{itemize}
The dependence graph $D(x)$ induces a \emph{vertex-labeled partial order} $P(x)= [\pos(x),\preceq,\lam]$, where $\preceq$ is the reflexive and transitive closure of $E(x)$; and we denote by $H(x)$ the \emph{\HD} of the induced partial order $P(x)$.
That is, $p\to q$ is an arc in $H(x)$ \IFF both, $p\to q$ is an arc in $E(x)$ and there is no position $r$ with $p\prec r \prec q$. 
The arcs in $H(x)$ are called \emph{\HA{s}}.
If a \HA $p\to q$ has the labeling 
$a=\lam(p)$ and $b=\lam(a)$, then, for better readability, we also sometimes write
$a\to b$ instead of $p\to q$. 

The set of positions (that is, the set of vertices) of $D(x)$, $P(x)$, and $H(x)$ are identical. 
Note that for every $1\leq i \leq |x|_a$ the $i$-th position labeled by a letter $a$ is well-defined as a position in the abstract graphs.
A standard result in trace theory due to Mazurkiewicz~\cite{maz77}, which also appears implicitly in the earlier work of Keller \cite{kel73}, says 
that $x = x'$ in $M(\Gam,I)$ \IFF $D(x) = D(x')$ \IFF $P(x) = P(x')$ \IFF $H(x) = H(x')$. 
A proof of this result can also be found in~\cite[Chapt.~1]{dr95} and~\cite{die90}.

\begin{defi}\label{def:indrel}
For $x\in \Gam^*$ we define its \emph{alphabet} $\gam(x)$ by 
$\gam(x)=\set{a\in \Gam}{|x|_a\geq 1}$ and, for $x,y\in \MGI$, we write $(x,y)\in I$ (and call
$x$ and $y$ 
\emph{independent}) if $\gam(x)\times \gam(y)\sse I$.
Thus, the independence relation~$I$ becomes an irreflexive and symmetric relation between traces. 
\end{defi}

A \emph{step} in $M(\Gam,I)$ is a product over pairwise independent letters;
that is, $s\in M(\Gam,I)$ is a step if it can be written as a product~$s=\prod\set{a_i}{i\in \os{1\lds m}}$ such that  
$(a_i,a_j)\in I$ for all $1\leq i <j \leq m$.
Note that $|s|$ cannot be greater than the size of a largest clique in $(\Gam,I)$.

As usual in graph theory, we say that a trace $x\in M(\Gam,I)$ is \emph{connected} if its dependence 
graph $D(x)$ (or its \HD $H(x)$), viewed as an undirected graph, is connected.
Equivalently,~$x$ is connected if its alphabet induces 
a connected subgraph in the undirected graph $(\Gam,D)$ 
where $D=\Gam\times \Gam\sm I$. The pair $(\Gam,D)$ is called a \emph{dependence alphabet}.\footnote{It is not `mad' to write $M(A,D)$ instead of 
$M(A,I)$ as $\gam:M(A,D) \to (A,D)$ is a graph morphism.}

\begin{defi}\label{def:minmax}
Let $x\in M(\Gam,I)$ be a trace. 
By $\min(x)$ (resp.~$\max(x)$) we mean the 
set of minimal (resp.~maximal) positions.
We write 
$\min_\lam(x)$ and~$\max_\lam(x)$ for the steps obtained by taking the product over the labels of the corresponding positions.\footnote{Note that if $x$ itself is a step, then $x= \min_\lam(x)=\max_\lam(x)$.}
\end{defi}

Every trace $x\in M(\Gam,I)$ can be written as 
a product $x=x^{(1)}\cdots x^{(d)}$ of traces such that
each $x^{(c)}$ is nontrivial and connected for all $1\leq c \leq d$ and $(x^{(c)},\,x^{(e)})\in I$ for all $1\leq c <e\leq d$.
The positions in each $x^{(c)}$ thus form a connected component in $D(x)$. 
As a consequence of this, it is easy to see that if $\min(x)\cap \max(x)\neq\es$, then either $|x|\leq 1$ or $x$ is disconnected.

The next proposition
characterizes factors in a trace $x$, to be exactly the convex subsets in the dependence graph of $x$. 
Recall that a subset of vertices $U$ in a directed graph is \emph{convex}
if all vertices of a directed path starting and ending in $U$ belong to~$U$, too. 

\begin{prop}\label{prop:factraze}
Let $D(x)$ be the dependence graph with vertex set $X$ of a trace $x\in M(\Gam,I)$,
and let $U$ be a subset of $X$.
Define further subsets $P$, $Q$, and $V$ of $X$ 
as follows: 
\begin{itemize}
\item $P=\set{p\in X\sm U}{\text{there is path from $p$ to some vertex in $U$}}$,
\item $Q=\set{q\in X\sm U}{\text{there is path from some vertex in $U$ to $q$}}$,
\item $V=X\sm (P\cup Q \cup U)$. 
\end{itemize}
Then the induced subgraphs in $D(x)$ of vertex sets $U,P,Q,V$ define traces $u,p,q,v\in M(\Gam,I)$ such that $(u,v)\in I$. 
Moreover, if $U$ is convex, then we have
$x=puvq$.
Vice versa, if there is a factorization $x=yuz$, then 
there there is partition $X=Y\cup U \cup Z$ in disjoint convex subsets such that the corresponding induced subgraphs 
are $y$, $u$, and $z$ respectively.  
\end{prop}

\begin{proof}
The proof is a standard routine by induction on $|X|$; see \eg~\cite[Prop.~1.2.8]{die90}.
\end{proof}

\subsection{Ochma{\'n}ski's Characterization of Recognizable Trace Languages}\label{sec:och845}

A complete deterministic automaton $\cA$ such that $L(\cA)\sse \Gam^*$ over an independence $(\Gam,I)$ 
is called~\emph{$I$-diamond}, if~$t \cdot ab = t \cdot ba$ whenever~$(a, b) \in I$ and $t$ is a state of $\cA$.\footnote{Recall that we write $t \cdot a$ to denote the unique state reachable from $t$ via a path labeled $a$.}
The diamond property is illustrated in the following picture.
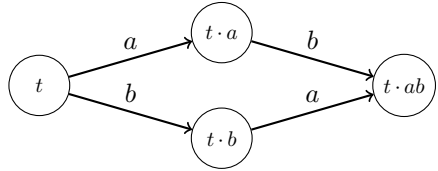
\begin{center} 
\vspace{-0cm}
		\begin{tikzpicture}[xscale=1.2, yscale=0.7, state/.style={circle, scale=0.8, draw, minimum size=1cm}]
		\node[state] (t) at (0,0) {$t$};
		\node[state] (ta) at (2,1) {$t\cdot a$};
		\node[state] (tb)  at (2,-1) {$t\cdot b$};
		\node[state] (tab) at (4,0) {$t\cdot a b$};
		\draw (t) edge[->,thick] node [above]{$a$}(ta);
		\draw (t) edge[->,thick] node [above]{$b$}(tb);
		\draw (ta) edge[->,thick] node [above]{$b$}(tab);
		\draw (tb) edge[->,thick] node [above]{$a$}(tab);
\end{tikzpicture}
\end{center}
If $\cA$ is~$I$-diamond, then $M(\Gam,I)$
acts on the right on the state space of $\cA$, and $t \cdot u$
is well-defined for all states $t$ and traces $u\in M(\Gam,I)$.
Thus, we can view $\cA$ as a complete deterministic $\MGI$-automaton, and such automata accept recognizable trace languages.

In general, the families of rational and recognizable sets do not coincide for trace monoids.
Indeed, if $(\Gam,I)$ is a finite independence alphabet and $a,b\in \Gam$ with $(a,b)\in I$, then 
\[
  (ab)^* = \set{a^n\,b^n}{n\in \N} \in \Rat(\MGI)\sm \Rec(\MGI). 
\]

The following characterization was obtained by Ochma{\'n}ski~\cite{och84,och85} in his PhD thesis; its proof can also be found in \cite[Chap.~4]{dr95} and \cite[Chap.~2]{die90}.
Ochma{\'n}ski's result relies on his notion of the \emph{concurrent-star operator}. He defined it for a trace language $L\sse \MGI$ as the following submonoid of $\MGI$, where all generators are connected:
\[
  \CoS{L}=\set{x\in \MGI}{\exists y\in \MGI: xy\in L \wedge (x,y)\in I \wedge x \text{ is connected}}^*.
\]

Thus, for $a,b\in \Gam$ with $(a,b)\in I$ as above, we obtain $\CoS{(ab)}= \os{a,b}^*\in \Rec(\MGI)$.
Note also that $L^*\sse \CoS{L}$ holds for every $L \sse \MGI$. 

\begin{thm}[Ochma{\'n}ski; 1984]\label{thm:ochm}
The following assertions hold.
\begin{enumerate}
\item The set $\Rec(\MGI)$ is the least family of subsets of $\MGI$ which contains all finite subsets and is closed under 
finite union, concatenation, and the concurrent-star operator.
\item We have $L\in \Rec(\MGI)$
  \IFF there is a $K\in \Reg(\Gam^*)$ with $K \sse \nflex (\MGI)$ such that $\theta(K)=L$.\footnote{Since $\nflex (\MGI) \in \Reg(\Gam^*)$, this is equivalent to $K = K' \cap \nflex (\MGI)$ for some $K' \in \Reg(\Gam^*)$.} 
In other words, the canonical \hom $\theta: \Gam^\ast \to \MGI$ induces a bijection between the regular languages consisting solely of lexicographic normal forms and the recognizable subsets of $\MGI$.
\end{enumerate}
\end{thm}

In the terminology of \cite{dm06}, this characterization means that $\Rec(\MGI)$ coincides with the 
family of \emph{normalized regular} subsets of $\MGI$.

\subsection{Transposition and Conjugacy}\label{sec:tracon}

As usual, we call elements $y$ and $z$ of an arbitrary monoid $M$ \emph{transposed} if $y=uv$ and $z=vu$ for some 
$u,v\in M$. 
The transposition relation is reflexive and symmetric, but not transitive in general (as, for example, in the trace monoid $\os{a,b,c}^*/\os{ab=ba}$). 
We denote by $\TO(x)$ the transposition orbit of $x \in M$; that is, $\TO(x)$ is the smallest subset of $M$ containing $x$ which is closed under transposition.

For every monoid $M$ there is also a notion of \emph{conjugacy}~$\sim$ defined by 
$y\sim z$ if $xy=zx$ for some $x\in M$. 
This relation is reflexive and transitive, but not symmetric in general. 

For every group~$G$, the transposition orbit~$\TO(g)$ coincides with the conjugacy class of~$g \in G$.  
In her thesis, Duboc~\cite{dub86,dub86tcs1} showed the analogous statement for trace monoids.

\begin{lem}[Duboc; 1986]\label{lem:tr_trans}
  Let $\MGI$ be a trace monoid.
  Then, for every $x \in \MGI$, the transposition orbit~$\TO(x)$ coincides with the conjugacy class of~$x \in \MGI$.
\end{lem}

\subsection{Lexicographical Normal Forms}

We conclude this section with the following lemma, which is a key ingredient for obtaining our later results.
It concerns the lexicographical normal forms of trace monoids. 
\begin{lem}\label{lem:trlex_per}
  Let $(\Sig, I)$ be an independence alphabet and $\Sig$ be a (possibly infinite) linearly ordered set and $u, p, v \in \MSI$. 
  By $\nflex:\MSI\to \Sig^*$ we denote the lexicographical normal form.\footnote{It is well-defined since every trace $x\in \MSI$ uses finitely many letters only.}
  Then for every $n\in \N$ there is some $N\in \N$ such that all 
$w\in \nflex(up^*v)$ which satisfy $|w|\geq N$ also satisfy $\exp(w)\geq n$. That is, $\nflex(up^*v)$ is periodically perfect.
\end{lem}
\begin{proof}
The assertion is trivial for $p=1$. 
Hence, we may assume that $p\neq 1$. 
As a consequence, we have $|up^nv|\geq n$ for all $n \in \mathbb{N}$ and, in particular, $up^*v$ is an infinite set in $\MSI$. 
If $\Sig$ is finite, then we let $\Gam=\Sig$. In the other case we define
$\Gam$ as the finite set 
$\Gam=\set{a\in \Sig}{|upv|_a\geq 1}$.
 
By~$\Pref(x)$ and~$\Suf(x)$, we denote the sets of all prefixes and suffixes of a trace $x \in \Gam^*$, respectively. 
Our proof also uses the notion of transposition orbit (\prref{sec:tracon}); and 
it utilizes an $I$-diamond finite deterministic $\Gam^*$-automaton~$\cA$ (\prref{sec:och845}) constructed as follows.
The set of states is $Q = \Suf(u) \times \TO(p) \times \Suf(v)$, with $(u, p, v) \in Q$ being initial, and all states being final.
For $a \in \Gam$, the transitions at $(r, q, s) \in Q$ are given as follows.

\begin{itemize}
\item If~$r = ar'$, then there is a transition~$(r, q, s) \overset{a}{\longrightarrow} (r', q, s)$.
\item If~$q = aq'$ and~$(a, r) \in I$, then there is a transition~$(r, q, s) \overset{a}{\longrightarrow} (r, q'a, s)$.
\item If~$s = as'$ and~$(a, rq) \in I$, then there is a transition~$(r, q, s) \overset{a}{\longrightarrow} (r, q, s')$.
\end{itemize}

By adding an additional non-final state~$\bot$ and appropriate transitions to it, we obtain a complete finite deterministic automaton $\cA$.
Moreover, this automaton is~$I$-diamond, as can easily be checked by examining the different types of transitions.
This means that we can view $\cA$ as an $\MGI$-automaton and, thus, $L(\cA)$ is a recognizable language of~$\MGI$.

Next, we want to determine the language accepted by~$\cA$.
Denote by~$L(r, q, s)$ the language accepted at the state~$(r, q, s) \in Q$.
Then~$w \in L(r, q, s)$ if and only if~$w \in \Pref(rq^{|w|}s)$.
The proof of this fact can be sketched as follows:
First, note that~$w \in \Pref(rq^{|w|}s)$ if and only if~$w \in \Pref(rq^{|w|}q's)$ for any~$q' \in M$ with~$\text{alph}(q') \subseteq \text{alph}(q)$.
Indeed, all the letters appearing in~$q'$ have some Hasse arc into the last~$q$ and thus~$w$ is too short too access any of them. Conversely, any letter of~$s$ that is accessible by~$w$ must be independent from~$q$ and thus also from~$q'$.
Using this fact, and a simple induction on~$|w|$, the proof now again boils down to a straight-forward case distinction corresponding to the three different types of transitions.

The characterization of~$L(\cA)$ now establishes the connection to~$\nflex(up^*v)$:
for each~$n \in \mathbb{N}$ the prefix of length~$n$ of~$\nflex(up^nv)$ is in the language~$\oi\theta (L(u, p, v)) = \oi\theta(L(\cA)) \sse \Gam^*$.
Moreover, the prefix itself is in normal form and thus contained in~$\nflex(L(\cA))$.
Since the automaton~$\cA$ is~$I$-diamond and~$\nflex(\MGI)$ is regular, we can employ the product automaton construction to obtain another trim deterministic automaton~$\tilde{\cA}$ accepting the language~$L(\tilde{\cA}) = \oi\theta(L(\cA)) \cap \nflex(M) = \nflex(L(\cA)) \sse \Gam^*$.
In particular, any path in~$\tilde{\cA}$ is labeled by a lexicographical normal form, as infixes of lexicographical normal forms are themselves in lexicographical normal form.
To complete our proof, it suffices to show that the language~$L(\tilde{\cA})$ is periodically perfect.

By \prref{thm:perper}, we must show that for every state~$z \in \tilde{Q}$ of~$\tilde{\cA}$, there is at most one simple cycle starting at~$z$. 
We denote by~$(\alpha(z), \pi(z), \beta(z)) \in Q$ the projection of the set of states of~$\tilde{\cA}$ onto the set of states of~$\cA$.
We are only interested in~$\pi(z)$, since~$\alpha(z)$ and~$\beta(z)$ must remain constant on any cycle by construction of~$\cA$.
Now consider the following diagram: 
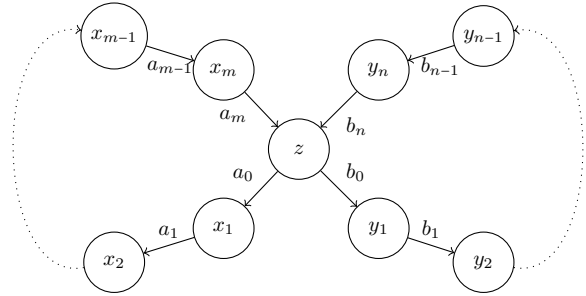
\begin{center}\begin{tikzpicture}[scale=0.3, state/.style={circle, draw, minimum size=1cm}]
\node[state, scale=0.8] (1) at (0, 0) {$z$};
\node[state, scale=0.8] (2a) at (-3.3, -3.5) {$x_1$};
\node[state, scale=0.8] (3a) at (-8.1, -5) {$x_2$};
\node[state, scale=0.8] (4a) at (-8.1, 5) {$x_{m-1}$};
\node[state, scale=0.8] (5a) at (-3.3, 3.5) {$x_m$};
\node[state, scale=0.8] (2b) at (3.5, -3.5) {$y_1$};
\node[state, scale=0.8] (3b) at (8.1, -5) {$y_2$};
\node[state, scale=0.8] (4b) at (8.1, 5) {$y_{n-1}$};
\node[state, scale=0.8] (5b) at (3.5, 3.5) {$y_n$};

\draw[->,thick] (1) to node[above left, scale=0.8] {$a_0$} (2a);
\draw[->,thick] (2a) to node[above, scale=0.8] {$a_1$} (3a);
\draw[->,thick] (4a) to node[below, scale=0.8] {$a_{m-1}$\;} (5a);
\draw[->,thick] (5a) to node[below left, pos=0.3, scale=0.8] {$a_m$} (1);
\draw[->,thick,dotted, bend left=100] (3a) to (4a);
\draw[->,thick] (1) to node[above right, scale=0.8] {$b_0$} (2b);
\draw[->,thick] (2b) to node[above, scale=0.8] {$b_1$} (3b);
\draw[->,thick] (4b) to node[below, pos=0.3, scale=0.8] {$b_{n-1}$} (5b);
\draw[->,thick] (5b) to node[below right, scale=0.8] {$b_n$} (1);
\draw[->,thick,dotted, bend right=100] (3b) to (4b);
\end{tikzpicture}\end{center}

  Note that if~$a_0 = b_0$ then~$x_1 = y_1$ by determinism of~$\cA$ so, by following the paths until they diverge, we may assume that~$a = a_0 < b_0$.
  Inductively, we will show that~$(a, b_i) \in I$ and that~$\pi(y_i) = ab_iv_i$ for some~$v_i \in \Gam^*$.
  This is certainly true for~$\pi(z)$: since we have outgoing transitions labeled~$a$ and~$b_0$, both of them must be minimal elements of~$\pi(z)= ab_0v_0$ and, in particular, independent from one another.
  Next, let~$\pi(y_i) = ab_iv_i$. Then~$\pi(y_{i+1}) = av_ib_i$ by construction of~$\cA$, and the outgoing label~$b_{i+1}$ must be a minimal element of~$\pi(y_{i+1})$.
  By the inductive hypothesis, the word~$b_0\cdots b_i a$ is not in lexicographical normal form and, therefore, we must have $b_{i+1} \neq a$.
  This yields~$(a, b_{i+1}) \in I$ as well as the desired factorization of~$\pi(y_{i+1})$.

  Finally, looking at the diagram again, we identify a path labeled~$b_0\cdots b_n a$ in~$\tilde{\cA}$.
  By the previous considerations, this path is not in lexicographical normal form~--~a contradiction!
  Hence, a pair of cycles as depicted in the above diagram cannot exist.
\end{proof}

\begin{rem}
  Before moving on, we remark that the automaton~$\cA$ constructed above does not accept the entire set~$\Pref(up^*v)$ in general.
  This set need not even be recognizable, as is the case for~$\Pref((ab)^*c) \sse \MGI$ where $\Gam = \os{a,b,c}$ and $I = \{(a, b), (b, a)\}$;
  indeed, we have $\oi\theta(\Pref((ab)^*c)) = \os{a,b}^* \cup \set{wc \in \os{a,b}^*c}{\abs{w}_a = \abs{w}_b} \sse \Gam^*$ and this is clearly not a regular language.
  However, our automaton accepts $L(\cA) = \{a, b\}^* \subseteq \Pref((ab)^*c) \sse \MGI$ and~$L(\tilde{\cA}) = a^*b^* \sse \Gam^*$, and this is sufficient to establish \prref{lem:trlex_per}.
\hspace*{\fill}$\diamond$
\end{rem}

\begin{exa}\label{ex:example_automaton}
  Consider~$M = M(\Gam, I)$ with~$\Gam = \{a, b, c, d\}$ and independence graph 
  
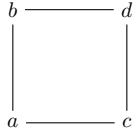
\begin{center}\begin{tikzpicture}[scale=0.3]
  \node[scale=0.8] (1) at (0, 0) {$a$};
  \node[scale=0.8] (2) at (0, 5) {$b$};
  \node[scale=0.8] (3) at (5, 0) {$c$};
  \node[scale=0.8] (4) at (5, 5) {$d$};

  \draw (1) to (2);
  \draw (2) to (4);
  \draw (1) to (3);
  \draw (3) to (4);
  \end{tikzpicture}\end{center}

  Now let~$u = b, v = c$ and~$p = abcd$.
  We thus have~$\nflex(up^nv) = ab(bc)^{n}cd(ad)^{n-1}$ for~$n \geq 1$.
  In particular, the set~$\nflex(up^*v)$ is not regular.
  However, the automaton~$\cA$ looks as follows. 
  
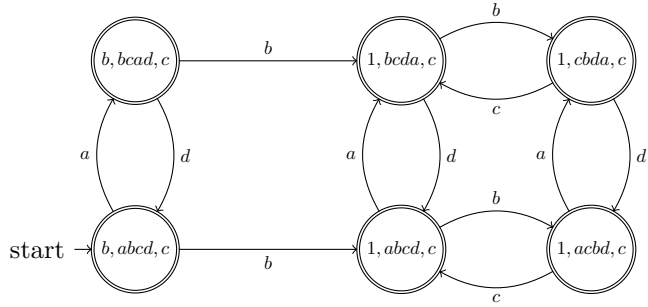
\begin{center}\begin{tikzpicture}[scale=0.5, state/.style={circle, draw, minimum size=1cm}]
    \node[state, initial, accepting, scale=0.7] (1a) at (-2, 0) {$b, abcd, c$};
    \node[state, accepting, scale=0.7] (2a) at (-2, 5) {$b, bcad, c$};
    \node[state, accepting, scale=0.7] (1b) at (5, 0) {$1, abcd, c$};
    \node[state, accepting, scale=0.7] (2b) at (5, 5) {$1, bcda, c$};
    \node[state, accepting, scale=0.7] (3b) at (10, 0) {$1, acbd, c$};
    \node[state, accepting, scale=0.7] (4b) at (10, 5) {$1, cbda, c$};

    \draw[->, bend left] (1a) to node[left, scale=0.7] {$a$} (2a);
    \draw[->, bend left] (2a) to node[right, scale=0.7] {$d$} (1a);
    \draw[->] (1a) to node[below, scale=0.7] {$b$} (1b);
    \draw[->] (2a) to node[above, scale=0.7] {$b$} (2b);
    \draw[->, bend left] (1b) to node[left, scale=0.7] {$a$} (2b);
    \draw[->, bend left] (2b) to node[right, scale=0.7] {$d$} (1b);
    \draw[->, bend left] (3b) to node[left, scale=0.7] {$a$} (4b);
    \draw[->, bend left] (4b) to node[right, scale=0.7] {$d$} (3b);
    \draw[->, bend left] (1b) to node[above, scale=0.7] {$b$} (3b);
    \draw[->, bend left] (3b) to node[below, scale=0.7] {$c$} (1b);
    \draw[->, bend left] (2b) to node[above, scale=0.7] {$b$} (4b);
    \draw[->, bend left] (4b) to node[below, scale=0.7] {$c$} (2b);
  \end{tikzpicture}\end{center}
  Here each state is labeled with the corresponding lexicographical normal forms and the additional state~$\bot$ is ommited for clarity. It is not hard to see that~$L(\cA) = \Pref(up^*)$, but already~$uv = bc \notin L(\cA)$.
  Now the automaton~$\tilde{\cA}$ may look as follows. 
  
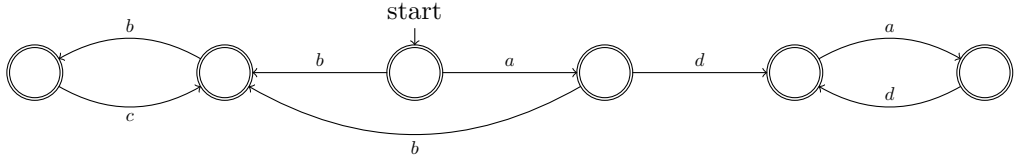
\begin{center}\begin{tikzpicture}[scale=0.5, state/.style={circle, draw, minimum size=1cm}]
    \node[state, initial above, accepting, scale=0.7] (1) at (0, 0) {};
    \node[state, accepting, scale=0.7] (2) at (5, 0) {};
    \node[state, accepting, scale=0.7] (3) at (10, 0) {};
    \node[state, accepting, scale=0.7] (4) at (15, 0) {};
    \node[state, accepting, scale=0.7] (5) at (-5, 0) {};
    \node[state, accepting, scale=0.7] (6) at (-10, 0) {};

    \draw[->] (1) to node[above, scale=0.7] {$a$} (2);
    \draw[->] (2) to node[above, scale=0.7] {$d$} (3);
    \draw[->, bend left] (3) to node[above, scale=0.7] {$a$} (4);
    \draw[->, bend left] (4) to node[above, scale=0.7] {$d$} (3);
    \draw[->] (1) to node[above, scale=0.7] {$b$} (5);
    \draw[->, bend right] (5) to node[above, scale=0.7] {$b$} (6);
    \draw[->, bend right] (6) to node[below, scale=0.7] {$c$} (5);
    \draw[->, bend left] (2) to node[below, scale=0.7] {$b$} (5);
  \end{tikzpicture}\end{center}

  Note that the automaton~$\tilde{\cA}$ is not unique, since it depends on the automaton for~$\nflex(M)$ used in the product construction.
  However, in any case the automaton accepts the periodically perfect language~$L(\tilde{\cA}) = (ad)^* \cup b(bc)^* \cup ab(bc)^*$, where indeed~$ab(bc)^* \in \nflex(\Pref(up^nv))$.
\end{exa}

\section{Graph Products}\label{sec:GPs}

The notion of a \emph{graph product} is a natural generalization
of the concept of a trace monoid. 

\subsection{Local and Global Monoids}\label{sec:locmon}

Let $(J,I)$ be a finite independence alphabet, which we view as an undirected graph. 
We think of the vertex set $J$ as an index set, and of each $i\in J$ as a \emph{color}. The edge set $I\sse J\times J$ is given as an irreflexive and symmetric relation.
We associate with each color $i\in J$ a nontrivial 
monoid $M_i$.
We call each monoid in the family $\set{M_i}{i\in J}$ a \emph{local monoid}.
Without restriction we assume that
$M_i\cap M_j=\os 1$ for all $i,j\in J$ with $i\neq j$. 
That is, local monoids share their neutral elements, but nothing else. 
Using this convention an element $x\in M_i\sm \os 1$ inherits the color $\chi(x) = i$ for all $i\in J$. 
Having this (and recalling that $I\sse J\times J$ is irreflexive), we define the associated \emph{global monoid}
as 
\begin{align}\label{eq:GPJ}
  \M={\GP\nolimits_{(J,I)}}\set{M_i}{i\in J} = 
  \ast_{i\in J}M_i /\set{xy=yx}{\exists (j,k)\in I: x\in M_j\wedge y\in M_k}.
\end{align}
Here, $\ast_{i\in J}M_i$ is the usual \emph{free product} of the monoids $M_i$. 
We fix for each local monoid $M_i$ an \epi $\pi_i: \Gam_i^*\to M_i$ such that $1\notin \pi_i(\Gam)$. 
This is no restriction, as letters in $\oi\pi_i(1)$ can simply be removed.
Using this, every letter $a\in \Gam_i$ obtains a color $\chi(a)=i\in J$.

The graph product $\M$ is generated by the disjoint union $\Gam=\bigcup\set{\Gam_i}{i\in J}$ and this defines an \epi $\pi: \Gam^*\to \M$.
Note that the graph product ${\GP_{(J,I)}}\set{M_i}{i\in J}$ is finitely generated \IFF every local monoid is finitely generated.\footnote{This is also restated as a part of \prref{prop:lift1}.}

Let $x\in \M$. A \emph{trace representation} of~$x$ 
is a  factorization $x= x_1\cdots x_m \in \M$ such that for all $1\leq i\leq m $ there is some $j\in J$ with  $x_i \in M_j\sm \os{1}$. The notation `trace' is used because the order of factors $x_ix_{i+1}$ can be interchanged, if $\big(\chi(x_{i}),\chi(x_{i+1})\big)\in I$. (This is formalized in 
\prref{def:Sig}.)
A trace representation of $x$ is called \emph{geodesic} if $m$ is minimal among all trace representations of~$x$. 
Using any geodesic trace representation $x = x_1\cdots x_m$, we define the length $|x|=m$
and the color $\chi(x)=\set{\chi(x_k)\in J}{1\leq k\leq m}$.

\begin{defi}\label{def:Sig}
Given the graph product $\M$ and the \epi $\pi:\Gam^*\to \M$ as above, we define the (typically infinite) alphabets 
$\Sig_i=M_i\sm \os 1$ for $i\in J$ and 
$\Sig = \bigcup\set{\Sig_i}{i\in J}$ (which is a disjoint union).
Moreover, we extend the irreflexive symmetric relation $I\sse 
J \times J$ to 
\[
  I_\Gam=\set{(a,b)\in \Gam \times \Gam}{\big(\chi((\pi (a)),\chi((\pi b))\big)\in I}
  \text{ and }
  I_\Sig=\set{(x,y)\in \Sig\times \Sig}{\big(\chi(x),\chi(x)\big)\in I}.  
\]
In order to simplify the notation we simply write $I$ for both,
$I_\Gam$ as well as $I_\Sig$.
\end{defi}

This convention defines free partially commutative monoids $\MGI$ and $\MSI$, and leads to various natural \epi{s}. 
\begin{align}\label{eq:natepi}
  \pi: \Gam^* \arc{\theta} M(\Gam,I) \arc{\phi} M(\Sig,I) \arc{\psi} \M  \,\text{ and }\,  \chi: \Gam^* \arc{\varphi\theta} M(\Sig,I) \arc{\eta} M(J,I)  \arc{\zeta} 2^J
\end{align}
The \epi $\theta$ is induced by the inclusion of~$\Gam$ in~$\MGI$.  
The \epi $\psi$ is induced by $\pi$, and the \epi $\zeta$
is induced the identity on $J$. More precisely, 
by $2^J$ we mean the finite commutative 
monoid $2^J=(2^J,\cup)$. Hence, if $i_1\cdots i_m$ is a sequence
of indices in $J$, then we have $\zeta(i_1\cdots i_m)=\set{i_k\in J}{1\leq k \leq m}$. 
Observe that $\zeta$ factorizes through the canonical \epi
of $\MJI \to \N^J$.
	
\begin{defi}\label{def:}
Let $u=x_1\cdots x_m\in \MSI$ be a trace where each 
$x_i$ is a letter in $\Sig$. 
Then we say that $u$ is \emph{reduced} and that $u$ is the \emph{reduced trace representation} of $\psi(u)\in \M$ if for all $1\leq i <j\leq m$ where  
$x_ix_j$ appears as a factor in the trace $u$ we have
$\chi(x_i)\neq \chi(x_j)$. 
\end{defi}

Note that every $x \in \M$ has a unique reduced trace representation $u \in \MSI$.
Moreover, if $u=x_1\cdots x_m\in \MSI$ as above, 
then $\psi(u) = \psi(x_1) \cdots \psi(x_m)$ is a geodesic trace representation in the sense of the informal definition preceding \prref{def:Sig}; hence, $|\psi(u)|=m$.

\subsection{Normal Forms}\label{sec:nfM} 

Suppose now, that every local monoid $M_i$ comes with both an \epi $\pi_i:\Gam_i^*\to M_i$ as above, and with a normal form $\nf_i:M_i\to \Gam_i^*$.
Furthermore, let us assume that $J$ is equipped with a linear order $\leq_J$.
We let $\Gam$ be the disjoint union $\Gam=\bigcup\set{\Gam_i}{i\in J}$ and $\pi: \Gam^*\to \M$
be the \epi as defined in \prref{sec:locmon}.
We add another mostly harmless requirement:\footnote{For a reference see \cite{ada92}.}
together with~$1 \notin \pi(\Gam)$, we assume the following.

\begin{conv}\label{conv:pinf}
  We have $\nf_i(1) = 1$ for all $i \in J$.
\end{conv}

As in \prref{sec:locmon}, we also associate with each letter $a \in \Gam_i$ the color $\chi(a) = i \in J$.
This induces a color $\chi(w) \in 2^J$ for all words $w \in \Gam^*$.

In order to define the normal form $\nfM:\M\to \Gam^*$, let $u = x_1 \cdots x_m \in \MSI$ be the reduced trace representation of $x = \psi(u) \in \M$, where each $x_i$ is a letter in $\Sig$.
Since $u$ is uniquely determined by $x$, mapping $x \mapsto u$ defines a normal form $\nf_\Sig : \M \to \MSI$.
From $\nf_\Sig$, we in turn obtain a normal form $\nf_\Gam:\M\to \MGI$ by replacing every letter $x \in \Sig_i$ (for all $i \in J$) in the trace $\nf_\Sig(x)$ by its corresponding normal form $\nf_i(a)\in \Gam_i^+ \sse \MGI$. 
Thus, when have $\nf_\Sig(x) = x_1 \cdots x_m$ where each $x_j$ is a letter in $\Sig$, then we obtain
\begin{align}
\label{eq:glnfI}
\nf_\Gam(x)&= \nf_{\chi(x_1)}(x_1)\cdots \nf_{\chi(x_n)}(x_n) \in M(\Gam,I)
\end{align}

We use the lexicographic normal $\nflex:\MGI\to \Gam^*$ based on the linear order $\leq_J$ on~$J$ to define the normal form $\nf_\M = \nflex \circ \nf_\Gam: \M \to \Gam^*$.
Since $\nflex$ requires a linear order on $\Gam$, we arbitrarily choose a linear order $\leq_i$ on $\Gam_i$ for all $i\in J$ and then combine these to a linear order~$\leq_\Gam$ on $\Gam$ by letting $a \leq_\Gam b$ if either $a\in \Gam_i$, $b\in \Gam_j$, and $i <_J j$ or if $a,b \in \Gam_i$ and $a \leq_i b$.
The normal form $\nflex: \MGI \to \Gam^*$ depends only on the linear order $\leq_J$, since for each $i\in J$ the letters in~$\Gamma_i$ are pairwise dependent.

\begin{rem}\label{rem:transn}
There are several transfer properties from the local normal forms
$\nf_i: M_i \to \Gam_i^*$ to the global normal form $\nf_\M$. 
For example, suppose that each $\Gam_i^*$ comes with a partial order $\preceq_i$ such that $\nf_i(x)\preceq_i u$ for all $u \in \oi\pi_i(x)$ and $x\in M_i$.
Then we can combine these, using the linear order $\leq_J$ on $J$, to obtain a partial order $\preceq$ on $\Gam^*$ and a normal form $\nfM(x)$ such that $\nfM(x)\preceq u$ for all $u\in \oi\pi(x)$ for all $x\in \M$.

A similar transfer property holds for geodesic local normal forms $\nf_i$, where $|\nf_i(x)|\leq |u|$ for all $u\in \oi\pi_i(x)\sse \Gam_i^*$. 
In this case, the global normal form $\nfM$ is also geodesic.
\hspace*{\fill}$\diamond$\end{rem}

Let us now turn to (perfectly) admissible normal forms in the sense of \prref{def:expinfty}.

\begin{thm}\label{thm:gpair}
  Let $\M = \GP_{(I, J)}\set{M_i}{i \in J}$ be a graph product where every local monoid~$M_i$ has an admissible pair $(\pi_i, \nf_i)$.
  Then the associated pair $(\pi, \nfM)$ for $\M$ is admissible.
  Moreover, if all pairs $(\pi_i, \nf_i)$ are perfectly admissible, then so is the pair $(\pi, \nfM)$.
\end{thm}
\begin{proof}
Let $u,p,v\in \M$. 
According to \prref{sec:bastt}, we may assume that $u,p,v$ are given by reduced traces in $M(\Sig,I)$ with $|p|\geq 1$. 
We represent the traces $u,p,v\in M(\Sig,I)$ by their \HD{s} where, as before, labeled vertices are called positions. 
Since each such position~$q$ has a label $\lam(q)\in \Sig$, we can define its color by $\chi(q)=\chi(\lam(q))$.

If the trace $p^2\in M(\Sig,I)$ is not reduced, then there is a \HA $s \to r$ in $\pos(p^2)$ where~$s \in \max(p)$ is a maximal position and~$r \in \min(p)$ is a minimal position (in the second factor $p$).
Letting $a = \lam(s)$ and $b = \lam(r)$, we have $a,b\in M_{i}\sm \os 1$ where $i = \chi(a) = \chi(b)$. 

Let $q \in M(\Sig, I)$ be the reduced trace defined by the positions in $\pos(p) \sm \os{r,s}$. 
Thus, we have $|q|\leq |p|-1$.
There are two cases to consider.
We begin with the case that $r = s$.
In this case, we must have $a = b$ and $(a, q) \in I$, since $r = s \in \min(p) \cap \max(p)$.
We divide this case into two subcases, according to whether or not the submonoid generated by $a$ is finite.
\begin{itemize}
\item 
If $|a^*|<\infty$, then there are $t, n\in \N$ with $n \geq 1$ such that $a^{t}=a^{t + n}$. 
We can then write
\[
  up^*v = uv \cup \dots \cup up^{t-1}v \cup up^t q^*v \cup \dots \cup up^{t+n-1}q^*v.
\]
Since $q$ has fewer letters than $p$, we can assume by induction that all sets $\nfM(up^{t+k}q^*v)$ with $0 \leq k < n$ are periodically nice (resp.\ periodically perfect).
Since $uv \cup \dots \cup up^{t-1}v$ is finite, this also holds for $\nfM(uv \cup \dots \cup up^{t-1}v)$.
As a finite union of such sets, $\nfM(up^*v)$ is periodically nice (resp.\ periodically perfect) so that we are done in this case.
\item 
If $|a^*|=\infty$, then the reduced representation of~$up^nv$ has some position with label~$xa^ny \in M_i$, where the~$x, y \in M_i$ each emerge from a potential reduction with a single maximal position in~$u$ and minimal position in~$v$, respectively.
Thus, $u = u'x$ and $v = yv'$.
By assumption, the set $\nf_i(xa^*y)$ is periodically nice (resp.\ periodically perfect).
Since $q$ uses fewer letters than $p$, we can assume by induction that the same holds for $\nfM(u'q^*v')$.
Also note that, by construction of the normal form $\nfM$, each word $w = \nfM(up^nv) = \nfM(ua^nq^nv)$ can be factored as $w = w_1w_2w_3$ where $w_2 = \nf_i(xa^ny)$ and $w_1w_3 = \nfM(u'q^nv')$.
It is now easy to see, that $\nfM(up^*v)$ is indeed periodically nice (resp.\ periodically perfect).
\end{itemize}

Thus, we are left to consider the case~$r \neq s$.
Let~$ab=c$ in~$M_i$.
Then we can write~$p=aqb$ with~$q \neq 1$. 
Again, we are done by induction, since~$up^+v = ua(qc)^*bv$ and~$|qc| < |p|$.

In the remainder of the proof, we may assume that~$p^2$ and hence~$p^n$ are reduced traces for all $n\in \N$.
Then let~$u'$ be the reduced trace representation of~$up^{|u|+1}$ and~$v'$ be the reduced trace representation of~$p^{|v|+1}v$.
Therefore~$\nf_\Sig(u'p^nv')=u'p^nv'$ is a reduced trace representation in $\MSI$ for all~$n \in \N$. 
Hence, by \prref{lem:trlex_per}, the set $\nflex(u'p^*v') \sse \Sig^*$ is periodically perfect.\footnote{Here, we use the lexicographical normal form $\nflex: \MSI \to \Sig^*$ defined using any linear order $\leq_\Sig$ on~$\Sig$ compatible with $\leq_J$ in the sense that $\chi(a) <_J \chi(b)$ implies $a <_\Sig b$ for all $a,b \in \Sig$.}
Moreover, it is contained in a finitely generated submonoid $\Del^* \sse \Sig^*$, so that the image $\varphi(\nflex(u'p^*v')) = \nfM(u'p^*v') \sse \Gam^*$ under the homomorphism $\varphi: \Sig^* \to \Gam^*$ defined by $a \mapsto \nf_i(a)$ for all $a \in \Sig_i$ remains periodically perfect in $\Gam^*$.
(Indeed, the restriction of $\varphi$ to $\Del^*$ increases the length of every word by at most a constant factor and, since $\phi$ maps nontrivial words to nontrivial words, $\phi$ cannot decrease the exponent of periodicity.)

The above argument shows that $\nfM(up^*v) \sse \Gam^*$ is periodically perfect, for it is the union of a finite set and the periodically perfect set $\varphi(\nflex(u'p^*v')) \sse \Gam^*$.
\end{proof}

The following corollary is the first step in our transfer results about infinite 
exponents of periodicity for \solu sets of quadratic equations from free groups to RAAGs.  
\begin{cor} \label{cor:graph_groups_perfect}
Let $(\Sig, I)$ be an independence alphabet where $\Sig$ has an \invol.
Further, let~$G=G(\Sig, I)$ be the associated free partially commutative group as defined in \prref{eq:GGamI}.
Then each short-lex normal form~$\nfslex : G \to \Sig^*$ of $G$ is perfectly admissible.  
\end{cor}

\subsection{Liftable Properties}\label{sec:trans}

Throughout this section $P$ denotes a property of finitely generated monoids, and by $\cL$ be a class of (necessarily finitely generated) monoids where $P$ holds. 
The monoids in $\cL$ are now the \emph{local monoids}. 
A graph product over finitely many monoids from $\cL$ is therefore called \emph{global}. 
Thus, every local monoid is global, but not the other way round. 
A property which holds for all global monoids (\wrt~$\cL$) is called a \emph{global property}.
In the following we collect some transfer results from local to global properties.
More formally:

\begin{defi}\label{def:lifties}
A property $P$ of finitely generated monoids which holds for all local monoids in $\cL$ is called \emph{liftable to graph products} (\wrt~$\cL$), if~$P$ also holds in every graph product 
$\M=\GP\nolimits_{(J,I)}\set{M_i}{i\in J}$ where $J$ is finite and we have $M_i\in\cL$ for all $i\in J$.

We also simply say that $P$ is \emph{liftable}.\footnote{If we do not specify the class $\cL$ explicitly, then we mean the 
class of all monoids where $P$~holds.} 
A liftable property $P$ is called \emph{faithfully liftable} if $P$ is a property which is is stable under taking finitely generated submonoids.\footnote{Liftability is a remote analogue to the \emph{Hasse principle in number theory.}}   
\end{defi}

The following proposition is well-known.
For every property discussed therein, the proof is straightforward and, therefore, omitted.
\begin{prop}\label{prop:lift1}
The following properties are faithfully liftable to graph products.
\begin{enumerate}
\item Being a trace monoid.
\item Being a group.
\item Being finitely generated.
\item Having a decidable word problem.\footnote{For us, the decidability of the word problem is defined for \fg~monoids, only.} 
\end{enumerate}
\end{prop}

Of course, many properties are not liftable; for example, the property to be finite. 
The following theorem generalizes the result on torsion freeness for graph products of groups, which was obtained by Green in her PhD~thesis~\cite[Thm.~3.26]{green1990graph}.
We show that torsion freeness is faithfully liftable to graph products over arbitrary monoids. 
\begin{thm}\label{thm:gptor}
Let $\cL$ be the class of all torsion-free monoids and ~$\M = \GP_{(J, I)}\set{M_i}{i \in J}$ be a graph product where all monoids $M_i$ are torsion free. 
Then $\M$ is torsion free.

More precisely, let~$y \in \M$;
if~$y^*$ is finite, then~$\TO(y)$ contains some~$z = z_1 \cdots z_n$ with $(z_i, z_j) \in I$ for all~$i \neq j$ where each~$z_i$ generates a finite submonoid~$z_i^* \leq M_{k_i}$ for some~$k_i \in J$.
\end{thm}
\begin{proof}
  Start by choosing some element~$z$ in the transposition orbit  $\TO(y)$ such that $z$ is of minimal length in~$\TO(y)$.
  In particular, the submonoid~$z^*$ is finite too, since~$z$ and~$y$ are related by a finite sequence of transpositions.
  We now show that~$z$ has the desired form.
  If~$z = z'z''$ decomposes nontrivially into the product of two independent elements~$(z', z'') \in I$, then we proceed by induction:
  both~$z'$ and~$z''$ are of minimal length in~$\TO(z')$ and~$\TO(z'')$ respectively; thus, they both possess a factorization of the desired form and we can recombine those to obtain one for~$z$.
  The base case is~$|z| \leq 1$, for which our claim is trivially satisfied.

  Therefore, interpreting~$z \in M(\Sigma, I)$ as a reduced trace, we assume by contradiction that~$z$ is connected and that~$|z| > 1$. Let $X = \pos(z)$ be the set of positions in its dependence graph.\footnote{Remember, $\pos(z)$ is the labeled vertex set of $D(z)$.} 
  Naturally, for~$z^*$ to be finite, a reduction must occur within some power~$z^n$.
  This always happens between a minimal position~$r \in \min(z)$ and a maximal position~$s \in \max(z)$ of two subsequent occurrences of~$z$ in~$z^n$, such that both~$\lambda(r), \lambda(s) \in M_i$ for some~$i \in J$.
  Since~$z$ is connected and~$|z| \geq 2$, these positions must be distinct: we have~$r \neq s$ and a factorization~$z = \lambda(r)v\lambda(s)$.
  This, however, implies~$z' = v\lambda(r)\lambda(s) \in \TO(z)$ with~$|z'| < |z|$, in contraction to how~$z$ was chosen in the first place.
\end{proof}

For our main result involving exponents of periodicity we also need to discuss the \fsqrtp in graph products. 
First of all, this property is not liftable. 
To see this, simply note that the free product $G = \Z / 2\Z \ast \Z / 2\Z$ (which is isomorphic to the semidirect product $\Z \rtimes \Z/2\Z$) has infinitely elements of order two; that is, $\abs{\sqrt[2]{1}\,} = \infty$ in $G$.
Therefore we need a slightly stronger assumption on our local monoids
in order to find a liftable property.
\begin{thm} \label{thm:gp_fsqr}
Let $\cL$ be the class of \fg~monoids having the \fsqrtp and, additionally, the property that $\sqrt[2]{1} = \os{1}$ holds for all $M\in \cL$.
Let~$\M = \GP_{(J, I)}\set{M_i}{i \in J}$ be a \fg~graph product such that
each $M_i$ belongs to $\cL$. Then these properties also hold in~$\M$.
Thus, the conjunction of both properties is 
liftable. 
\end{thm}
\begin{proof}
Let $w,x\in \MSI$ be reduced traces such that
$\psi(x)\in \sqrt[2]{\psi(w)}$. Define $X=\pos(x)$ to be the set of positions in its dependence graph, and $\preceq$ be the induced partial order which is defined, as above, by the reflexive transitive closure of the arc relation in the dependence graph.  We will first describe the reduction process of $x^2$ to $w$.
To this end, we consider two sets $Y,Z \subseteq X$ defining a prefix $y$ and a suffix $z$ of $x$, respectively, such that $\psi(yz) = \psi(w)$. 
We start with~$Y = X = Z$, which implies $X=Y\cup Z$ and $Y\cap Z\neq \es$.

If there now exist $p \in \max(Y)$ with label $\lambda(p) = a$ and $q \in \min(Z)$ with label $\lambda(q) = b$ such that $ab=1$, then we proceed with 
$Y' = Y\setminus \os{p}$ and $Z' = Z\setminus\os{q}$.
Note, that necessarily $p \neq q$, since $p=q$ implies $a=b$ and thus we would have $a^2 = 1$, which is excluded by assumption.
Thus, we still have $X=Y'\cup Z'$.
Moreover, the construction implies $q \prec p$ and, since $x$ is reduced, the existence of a position $r \in X$ such that $q \prec r \prec p$. Thus, $Y'\cap Z'\neq \es$.

Thus, after at most~$|x|/2$ repetitions this procedure terminates with sets $Y, Z \subseteq X$, satisfying~$X = Y\cup Z$ and $Y \cap Z \neq \emptyset$, such that no more suitable $p\in \max(Y)$ and $q\in \min(Z)$ exist.
The corresponding prefix $y$ and suffix $z$ of $x$ still satisfy $\psi(yz) = \psi(w)$.
Moreover, the product $yz$ is almost reduced: every label of a non-maximal position in $Y$ is the label of a position of $w$ and the same holds for every non-minimal position in $Z$. 
On the other hand, some $p \in \max(Y)$ may combine with some $q \in \min(Z)$ to form a position $r \in \pos(w)$, yielding the equation $\lambda(p)\lambda(q) = \lambda(r)$.
By construction, if~$p \neq q$, then both~$p \in Z\setminus\min(Z)$ and~$q \in Y\setminus\max(Y)$ as well; that is, their labels are already determined by the labels of the corresponding positions in~$\pos(w)$.
The remaining case is thus~$p = q$.
By assumption on the local monoids~$M_i$, this only leaves finitely many choices for the label, as it is a solution of~$\lambda(p)^2 = \lambda(r)$.
Moreover, the number of positions of~$x$ is bounded by the number of positions of~$w$.
That~$\psi(x^2) = 1$ is impossible unless~$x=1$ follows by the same argument. 
\end{proof}

As a direct consequence of \prref{thm:gptor} and \prref{thm:gp_fsqr} we obtain the following.

\begin{cor}\label{cor:gp_fsqr}
The property given by the conjunction of being torsion free and also having the \fsqrtp is liftable to graph products.
\end{cor}

\section{Quadratic Equations}\label{sec:qwe}

As an application of the previous sections, we present a result on the solution sets of quadratic equations over groups with certain properties.
Suppose we are given a group $G$, and let $\cX = \cX^+ \cup \cX^-$ be a set of variables with a free involution $X \mapsto X^{-1}$ exchanging the elements of $\cX^+$ and $\cX^-$.
We denote by $F(\cX)$ the free group with basis $\cX^+$, and by $G[\cX] = G * F(\cX)$ the corresponding free product.
The set of variables $\cX$ is then viewed as a subset of $F(\cX)$.

An element~$W \in G[\cX]$ is called an \emph{equation} over $G$, and a finite subset~$\cS \subseteq G[\cX]$ is called a (finite) \emph{system of equations} over $G$.
By~$|\cS|_X = \sum_{W \in \cS}|W|_X$ we denote for~$X \in \cX$ the number of occurrences of $X \in \cX$ in the reduced words representing the equations $W \in \cS$.
The system~$\cS$ is called \emph{quadratic} provided that~$|\cS|_X + |\cS|_{X^{-1}} \leq 2$ for every~$X \in \cX$.

A \emph{solution} of such a system is given by a homomorphism~$\sigma: G[\cX] \to G$ with~$\sigma|_G = \id_G$ and~$\sigma(W) = 1$ for every equation~$W \in \cS$.
Sometimes it can be useful to additionally constrain the values of~$\sigma(X)$ for each $X \in \cX$ to some subset $L_X \sse G$, where $L_{X^{-1}} = (L_X)^{-1}$.
We call such a family~$L = (L_X)_{X \in \cX}$ of subsets a \emph{recognizable constraint} if each~$L_X\in \Rec(G)$, and we say that a solution $\sig$ satisfies this constraint if $\sig(X) \in L_X$ for all $X \in \cX$.
Equivalently, such a constraint can be given in the form of a homomorphism $\mu : G \to H$ onto a finite group~$H$ and subsets $\mu_X \sse H$, where $L_X = \oi\mu(\mu_X)$.
Accordingly, a \emph{solution} of $(\cS, \mu)$ is a solution $\sig$ of the system of equations $\cS$ which satisfies $\mu(\sig(X)) \in \mu_X$ for all $X \in \cX$.

Finally, we define the \emph{(full) solution set} of a system of equations with recognizable constraints~$(\cS, \mu)$ as the set of all possible images of variables under solutions, that is, 
\[
  \Sol(\cS, \mu) = \set{\sigma(X)}{\text{$\sigma$ is a solution of~$(\cS, \mu)$ and $X \in \cX$}} \sse G. 
\]
The following theorem is our main result about infinite
\solu sets of quadratic systems of equations over groups and the exponent of periodicity. 

\begin{thm}\label{thm:snN}
  Let $G$ be a finitely generated torsion-free group with the \fsqrtp, and $\nf : G \to \Gam^*$ an admissible normal form.
  Then, for every finite quadratic system of equations with recognizable constraints~$(\cS, \mu)$, its solution set satisfies
  \[
    \abs{\Sol(\cS, \mu)} = \infty \text{ \IFF\ } \exp(\nf(\Sol(\cS, \mu))) = \infty.
  \]
\end{thm}
\begin{proof}
Let us call $(\cS, \mu)$ \emph{nice} if $|\Sol(\cS, \mu)| < \infty$ or $\exp(\nf(\Sol(\cS, \mu)))=\infty$; that is, if $\nf(\Sol(\cS, \mu))$ is a periodically nice subset of~$\Gamma^*$.
Thus, our goal is to prove that every quadratic system of equations with recognizable constraints is nice.

If $G$ is finite, then the assertion is trivial.
Therefore, we assume that $G$ is infinite and, hence, whenever we consider an \epi $\mu:G \to H$ onto a finite group~$H$, then $|\oi\mu(h)|=\infty$ for all $h \in H$.
Let $(\cS, \mu)$ be a system of quadratic equations with recognizable constraints where, without restriction, every $W \in \cS$ is a cyclically reduced word in~$G[\cX]$.

Without loss of generality, we also assume that $\abs{\Sol(\cS, \mu)} = \infty$.
Note that, whenever we write $\Sol(\cS, \mu)$ as a finite union $S_1 \cup \dots \cup S_n$, then at least one of the $S_i$ is infinite and it suffices to prove that $\exp(\nf(S_i)) = \infty$ for one such $S_i$.
Hence, it suffices to consider the case of recognizable constraints $\mu : G \to H$ where the set $\mu_X \sse H$ is a singleton for each $X \in \cX$.
Given this, we can extend $\mu$ to a homomorphism $\mu : G[\cX] \to H$ and then a solution $\sig$ of $\cS$ is a solution of $(\cS, \mu)$ \IFF $\mu(\sig(X)) = \mu(X)$ for all $X \in \cX$.

Let us also note that, for each variable $X \in \cX$, the set $L_X = \oi\mu(X) \cap G$ is a rational subset of $G$ (since $\Rec(G) \sse \Rat(G)$ as $G$ is finitely generated).
Hence, $\nf(L_X) \sse \Gam^*$ is periodically nice by \prref{lem:pumping_lemma}.
In particular, $\exp(\nf(L_X)) = \infty$ as $\abs{\nf(L_X)} = \infty$.

\medskip

Let us now fix some solution $\sig_0$ of $(\cS, \mu)$.
Our goal is to construct further solutions~$\sigma_n$ satisfying~$\exp(\nf(\sigma_n(X))) \geq n$ for some~$X \in \cX$.
We consider the following seven cases.

\begin{enumerate}
  \item Some variable does not occur, that is, we have~$|\cS|_X + |\cS|_{X^{-1}} = 0$ for some variable~$X \in \cX^+$.
  In this case we can freely change the value $\sig_0(X)$ to any other value $x \in L_X$.
  Thus, we obtain a solution $\sig_n$ of $(\cS, \mu)$ with the desired properties by simply setting $\sig_n(X) = x$ for some $x \in L_X$ with $\exp(\nf(x)) \geq n$ and $\sig_n(Y) = \sig_0(Y)$ for all $Y \in \cX^+$ with $Y \neq X$.

  \item Some variable occurs in separate equations, that is, $\abs{U}_X + \abs{U}_{\oi X} = \abs{V}_X + \abs{V}_{\oi X} = 1$ for some variable~$X \in \cX^+$ and distinct $U,V \in \cS$.
    Inverting and conjugating these equations if necessary, we can write $U = U'X$ and $V = \oi XV'$ for some $U', V' \in G[\cX \sm \os{X, \oi X}]$.
    Replacing $U$ and $V$ by the single equation $W = U'V'$ then allows us to eliminate $X$ from the system of equations.
    In this way, we obtain a new system $(\cS', \mu')$ containing fewer variables.
    Moreover, there is a one-to-one correspondence between solutions $\sig'$ of $(\cS', \mu)$ and solutions $\sig$ of $(\cS, \mu)$ via $\sig(X) = \sig'(U')$.
    Thus, $(\cS, \mu)$ is nice \IFF the new system $(\cS', \mu')$ is so, and we can assume this to be the case by induction on $\abs{\cX}$.

  \item There is an independent equation, meaning that, for some $U \in \cS$ and all $X \in \cX^+$, it holds that $\abs{U}_X + \abs{U}_{\oi X} \geq 1$ implies $\abs{V}_X + \abs{V}_{\oi X} = 0$ for all $V \in \cS$ with $U \neq V$.
    In this case, we can split $(\cS, \mu)$ into two systems $(\cS', \mu)$ and $(\cS'', \mu)$ with $\cS' = \os{U}$ and $\cS'' = \cS \sm \os{U}$, where $(\cS', \mu)$ is a system over the variables occurring (positively or negatively) in $U$ and the second system, $(\cS'', \mu)$, is a system over the remaining variables.
    By induction on~$\abs{\cX}$, we may once again assume that $(\cS', \mu)$ and $(\cS'', \mu)$ are nice.
     Due to the obvious correspondence between solutions of $(\cS, \mu)$ and pairs of solutions of $(\cS', \mu)$ and $(\cS'', \mu)$, we have $\Sol(\cS, \mu) = \Sol(\cS', \mu) \cup \Sol(\cS'', \mu)$.
     Hence, $(\cS, \mu)$ is also nice.
\end{enumerate}

By these first three cases, we may now assume that our system consists of a single quadratic equation~$W$ and that each variable~$X \in \cX$ occurs at least once in~$W$.

\begin{enumerate} \setcounter{enumi}{3}
  \item Some variable occurs only once, that is, we have $\abs{W}_X + \abs{W}_{\oi X} = 1$ for some $X \in \cX$.
  Similar to case (2), we may remove the variable $X$ and the (only remaining) equation~$W$ from the system, leaving us with the empty system of equations $(\emptyset, \mu)$ over $\cX \sm \os{X, \oi X}$.
  Thus, $\Sol(\cS, \mu)$ contains $\Sol(\emptyset, \mu) = \bigcup\set{L_Y}{Y \in \cX \sm \os{X, \oi X}}$ and, in particular, we therefore have $\exp(\nf(\Sol(\cS, \mu))) \geq \exp(\nf(L_Y)) = \infty$ using some $Y \in \cX \sm \os{X, \oi X}$.\footnote{At least one such variable $Y$ exists, for otherwise $\Sol(\cS, \mu)$ would be finite.}

  \item Some variable occurs both positively and negatively, meaning that $\abs{W}_X = \abs{W}_{\oi X} = 1$ for some $X \in \cX^+$.
    Thus, up to conjugation, $W = X U \oi X V$ with $U, V \in G[\cX\setminus\{X, \oi X\}]$.
  If~$\sigma_0(U) = 1$, then we choose an element~$x \in L_X$ satisfying~$\exp(\nf(x)) \geq n$ and obtain a solution~$\sigma_n$ by defining~$\sigma_n(X) = x$ and~$\sigma_n(Y) = \sigma_0(Y)$ for all $Y \in \cX^+$ with $Y \neq X$.

  In case~$\sigma_0(U) \neq 1$, let $k \geq 1$ be the order of $\mu(U)$ in the finite group $H$ and consider the automorphism $\tau : G[\cX] \to G[\cX]$ defined by $\tau(X) = XU^{k}$, leaving all $Y \in \cX^+$ with $Y \neq X$ and $G$ invariant.
  By construction, $\tau(W) = W$ and $\mu\tau = \mu$ so that precomposition with $\tau$ maps   solutions of $(\cS, \mu)$ to solutions of $(\cS, \mu)$.
  For the solution $\sig_N = \sig_0\tau^N$, we obtain the image $\sig_N(X) = \sig_0(X) \sig_0(U)^N \in G$, which is different for distinct values of $N$ as $G$ is torsion free and $\sig_0(U) \neq 1$. 
  Hence, $\exp(\nf(\Sol(\cS, \mu))) \geq \exp(\nf(\sig_0(X) \sig_0(U)^*)) = \infty$.

  \item There are two variables that each occur twice with the same exponent, meaning that we have $\abs{W}_X = \abs{W}_Y = 2$ for distinct $X, Y \in \cX$.
  Upon applying an automorphism defined by $Z \mapsto \oi Z$ for some or both $Z \in \{X, Y\}$ if necessary, we may assume that $X,Y \in \cX^+$.
  Up to conjugation, there are two possibilities for the equation:
  \[
    W = XUYV_1XV_2YV_3 \text{ or } W = XUYV_1YV_2XV_3 
  \]
  where $U,V_1,V_2,V_3 \in G[\cX \sm \os{X, \oi X, Y, \oi Y}]$.
  The automorphism $\tau: G[\cX] \to G[\cX]$ defined by $\tau(X) = X \oi{(UY)}$ yields $\abs{\tau(W)}_Y = \abs{\tau(W)}_{\oi Y} = 1$ in either case.
  Consider the system $(\cS', \mu')$ with $\cS' = \os{\tau(W)}$ and $\mu'$ given by $\mu'(X) = \mu(XUY)$ and $\mu'(Z) = \mu(Z)$ for all $Z \in \cX^+$ with $Z \neq X$.
  By the preceding case, $(\cS', \mu')$ is nice and, more specifically, it admits solutions $\sig'_n$ with $\exp(\nf(\sig'_n(Y))) \geq n$ for all $n \in \N$.
  But then we are done, since the mapping $\sig_n = \sig'_n \tau$ is a solution of $(\cS, \mu)$ with $\sig_n(Y) = \sig'_n(Y)$.

  \item There is only one variable and it occurs twice with the same exponent, that is, $W = X u X v$ with $\cX = \{X, \oi X\}$ and $u, v \in G$ after a suitable conjugation.
  In this case, every solution~$\sig$ of $(\cS, \mu)$ is determined by its image $x = \sig(X) \in G$, which must of course satisfy $xuxv = 1$ or, equivalently, $(xu)^2 = \oi v u$.
  Since $G$ has the \fsqrtp, we deduce that there are only finitely many solutions of $(\cS, \mu)$, contrary to our assumption.
\end{enumerate}

As the above list of cases is exhaustive, this concludes the proof.
\end{proof}

\begin{cor}\label{cor:ggpqwe}
Let $\G=\GP_{(J,I)}\set{G_i}{i\in J}$ be a finitely generated graph product of groups, where each $G_i$ has a normal form $\nf_i : G_i \to \Gam_i^*$, and suppose that the following hold.
\begin{enumerate}
\item For every $g \in G_i$ the set $\sqrt[2]{g} = \set{h \in G_i}{h^2 = g}$ is finite.
\item For all $u,p,v \in G_i$ with $p \neq 1$ we have $\exp(\nf_i(up^*v)) = \infty$.
\end{enumerate}
Let $\nf_\G : \G \to \Gam^*$ be the associated normal form as constructed in \prref{sec:nfM} (for graph products of monoids).
Then, for every quadratic system $(\cS, \mu)$ of equations with recognizable constraints over $\G$, either the full solution set $\Sol(\cS, \mu)$ is finite or $\exp(\nf_\G(\Sol(\cS, \mu))) = \infty$.
\end{cor}
\begin{proof}
Item~$(2)$ implies that every nontrivial cyclic subgroup of a local group $G_i$ is infinite.
In other words, each $G_i$ is torsion free.
Since the local groups also have the \fsqrtp, we see that $\G$ is both torsion free and has the \fsqrtp by \prref{cor:gp_fsqr}.
The associated normal form $\nf_\G$ is admissible by \prref{thm:gpair}.
\end{proof}

It remains to show there are interesting torsion-free groups $G$ which can be 
written as a graph product $\smash{\GP_{(J,I)}\set{G_i}{i\in J}}$ so that we can apply \prref{cor:ggpqwe}.
This is done in the next section.
Our first example is, not surprisingly, the family of RAAGs.\footnote{Encore et encore un RAAG, cet obscur objet du désir comme l'avait dit Luis Bu{\~n}uel (1900 -- 1983)}

\section{Applications of \prref{cor:ggpqwe}}\label{sec:appl}

\subsection{Free Partially Commutative Groups}\label{sec:raag}

Our first application covers the finitely generated torsion-free free partially commutative groups; that is, the family of RAAGs (aka.\ graph groups). 
In this case, each local group is the infinite cyclic group $\Z$ (with its usual set of monoid generators and its usual normal form). 
\prref{cor:ggpqwe} yields the following.\footnote{Remarkably, it is open whether the corresponding assertion for trace monoids holds.}
\begin{cor}\label{cor:raag}
    Let~$\G = G(\Gam, I)$ be a finitely generated free partially commutative group, and let $(\cS, \mu)$ be a quadratic system of equations over~$\G$ with recognizable constraints.
Then either $\Sol(\cS, \mu)$ is finite or $\exp(\nfslex(\Sol(\cS, \mu)))=\infty$ for every short-lex normal form $\nfslex$.   
\end{cor}

\subsection{Baumslag-Solitar Groups}\label{sec:BS}

The \BSG{s} constitute one of the most prominent families of groups and they have been widely studied. 
For any pair of integers~$(p,q)$ with $1\leq p \leq |q|$, the Baumslag-Solitar group $\BS(p,q)$ is defined  as the one-relator group 
\begin{align}\label{eq:BSG}
\BS(p,q) = \langle a,t \mid ta^p\oi t = a^q \rangle.
\end{align}

From this presentation, it is apparent that $\BS(p,q)$ is an HNN extension of $\Z$ with respect to the unique isomorphism $p\Z \to q\Z$ mapping $p$ to $q$.
As with all HNN extensions of a torsion-free group, the Baumslag-Solitar groups are torsion free.

We let $\Gam=\os{a,\oi a, t,\oi t}$ and $\pi:\Gam^*\to \BS(p,q)$ the canonical \epi which is induced by the inclusion $\Gam\sse \BS(p,q)$.
In this setting there exists a string-rewriting system (aka.\ semi-Thue system)~$S\sse \Gam\times \Gam$ for~$\BS(p,q)$ given by 
the following rules.\footnote{For the notation and results about string rewriting system we refer to the textbook \cite{bo93springer}. A terminating and confluent rewriting system is also called \emph{convergent}.}
\[
    \begin{array}{rcll}
 a^{kq + r}t&\to&a^rta^{kp} & \text{for $k \in \Z\setminus\{0\}, 0 \leq r < |q|$;}\\
 a^{kp + r}t^{-1}&\to&a^rt^{-1}a^{kq} & \text{for $k \in \Z\setminus\{0\}, 0 \leq r < |p|$;}\\
 tt^{-1}, t^{-1}t&\to&1\;. & 
    \end{array}
\]
The system $S$ is terminating and confluent, and  
as usual, the system~$S$ thereby defines a normal form, which we denote by~$\nf_S$,
such that $\nf_S(\BS(p,q))$ is the set of words where no rule of $S$ can be applied. 
There is another important rewriting system~$B\sse \Gam\times \Gam$ called Britton reduction.
It is given by the rules \[
    \begin{array}{rcll}
 ta^{kp}t^{-1}&\to&a^{kq} & \text{for $k \in \Z$;}\\
 t^{-1}a^{kq}t&\to&a^{kp} & \text{for $k \in \Z$.}\\
    \end{array}
\] 

The system~$B$ is not convergent, but 
it allows us to speak of cyclically Britton reduced words~--~that is words whose transpositions are all Britton reduced.
Moreover, it is not hard to see that every~$B$-rule can be realized by a sequence of~$S$-rules and, in particular, any normal form with regard to~$S$ is Britton reduced.
In general, every Britton reduced word~$w$ may be written as 
\[
    w = a^{n_1}t^{\eps_1}a^{n_2}\cdots t^{\eps_m}a^{n_{m+1}}
\] 
with~$\varepsilon_i \in \{\pm 1\}$.
The sequence~$\ts(w) = (\varepsilon_1, \dots, \varepsilon_m)$ is then called the~$t$-sequence of~$w$ and it can be shown that the~$t$-sequences for~$w$ and~$\nf_S(w)$ coincide.
\begin{prop} \label{prop:bsg_perfect}
    Let~$\BS(p, q)$ be a Baumslag-Solitar group and $u, w, v \in \BS(p, q)$ with $w \neq 1$. Then we have $\exp(\nf_S(uw^*v))=\infty$. In other words, the normal form~$\nf_S$ is perfectly admissible. 
\end{prop}
\begin{proof}
First note, that the normal form of~$x \in \BS(p,q)$ can be uniquely written as
\[
\nf_S(x)  =  a^{r_1}t^{\eps_1}a^{r_2}\cdots a^{r_m}t^{\eps_m}a^{\beta(x)}
\]
    with~$\beta(x) \in \Z$ and such that~$\eps_i = -1$ implies~$0 \leq r_i < p$, and~$\eps_i = 1$ implies~$0 \leq r_i \leq |q|$.
Having this, we can also define $\alp(x)$ by 
\[
\alpha(x)= \nf_S(x) a^{-\beta(x)}
\]
    The definition of~$\alpha$ and~$\beta$ is motivated by the following observation:
    Given elements~$u, v \in \Gam^*$ such that the product~$uv$ is Britton reduced, we observe the following relationship.
\begin{align}\label{eq:nfab}
\nf_S(uv) = \alpha(u)\nf_S(a^{\beta(u)}v)
\end{align}
We start with the situation where~$\nf_S(u)\nf_S(w)^n\nf_S(v) \in \Gam^*$ is Britton reduced for every~$n \in \N$.
    Thus, the previous observation applies. 
    Letting~$k_0 = \beta(u)$ and~$k_{i+1} = \beta(a^{k_i}w)$ for $0\leq i <n$, we inductively obtain \[
\nf_S(uw^nv) = \alpha(u)\alpha(a^{k_0}w)\alpha(a^{k_1}w) \cdots \alpha(a^{k_{n-1}}w) \alpha(a^{k_n}v)a^{\beta(a^{k_n}v)}
    \]
    Thus, if the sequence numbers~$k_i$ becomes periodic, so does the sequence of words~$\alpha(a^{k_i}w)$ and we obtain the desired claim.
    Otherwise, we must have~$\lim_{n \to \infty}|k_n| = \infty$.
    In particular, this also implies~$\lim_{n \to \infty}|\beta(a^{k_n}v)| = \infty$ and hence the claim follows again.

    We still have to deal with the case that~$\nf_S(u)\nf_S(w)^n\nf_S(v) \in \Gam^*$ is not Britton reduced. In this case we proceed as follows.
    First, we find some conjugate~$\tilde{w} \in \BS(p, q)$ of~$w = x^{-1}\tilde{w}x$ such that~$\nf_S(\tilde{w})$ is cyclically Britton reduced; then, in particular,~$\nf_S(\tilde{w})^n$ is cyclically Britton reduced for all~$n \in \N$.
    We also set~$\tilde{u} = ux^{-1}$ and~$\tilde{v} = xv$ and thus~$uw^*v = \tilde{u}\tilde{w}^*\tilde{v}$. 
    Finally, by choosing~$n_0 \in \N$ sufficiently large, we obtain~$\tilde{u}' = \tilde{u}\tilde{w}^{n_0}$ and~$\tilde{v}' = \tilde{w}^{n_0}\tilde{v}$ such that both~$\nf_S(\tilde{u}')\nf_S(\tilde{w})$ and $\nf_S(\tilde{w})\nf_S(\tilde{v}')$ are Britton reduced.
    
    We thus find that~$\nf_S(\tilde{u}')\nf_S(\tilde{w})^n\nf_S(\tilde{v}')$ is Britton reduced thus the claim holds for the language~$\nf_S(\tilde{u}'\tilde{w}^*\tilde{v}')$ by the previous considerations.
 Now, the set~$\nf_S(uw^*v)$ is the union of this language with a finite set (which depends on $n_0$), so we are done. 
\end{proof}
\begin{rem}
    A similar proof strategy can be employed to proof an analogous statement for arbitrary graphs of groups, which admit a similarly constructed normal forms.
    However, we need a very strong assumption on the normal forms~$\nf_v$ of the vertex groups~$G_v$:
    The set~$\nf_v(G_v)$ must be periodically perfect for every vertex~$v \in V$.
    As expected, this is the case for the usual normal form of the infinite cyclic group $\Z$.
\hspace*{\fill}$\diamond$
\end{rem}

We are therefore only left to deal with the finite square roots property.
\begin{thm}\label{thm:bsgfsq}
    Let~$\BS(p, q)$ be the Baumslag-Solitar group with $1\leq p \leq |q|$.
    Then~$\BS(p, q)$ has the \fsqrtp if and only if the following two assertions hold:
    \begin{enumerate}
\item The sum of the parameters is nonzero: we have $p+q\neq 0$.
\item Either $p=1$ or $p$ and $q$ are odd (or both). 
    \end{enumerate}
    Moreover, if the \fsqrtp holds, then square roots are unique and, in particular,~$\sqrt[2]{1} = \os{1}$.
\end{thm}
\begin{proof}
We make a case distinction according to the parameters~$1\leq p \leq |q|$. 
  We start with three cases where~$\BS(p, q)$ has the \fsqrtp.
\begin{enumerate}
\item \textbf{We have~$p=1$ and~$q \neq -1$.}
In this case,~$\BS(1, q)$ is isomorphic to the semi-direct product~$\mathbb{Z}[1/q] \rtimes_\alpha \mathbb{Z}$ where the group action of~$\mathbb{Z}$ on~$\mathbb{Z}[1/q]$ is given by~$\alpha_k(x) = q^kx$.
Thus, we have~$(x, k)^2 = (y, m)$ if and only if~$k = m/2$ and~$x + q^kx = y$.
  Since~$1 + q^k \neq 0$, we get a unique solution~$x = y/(1 + q^k) \in \Q$,  and therefore at most one solution in~$\mathbb{Z}[1/q]$.

\item \textbf{We have~$p = q$ and $p$ is odd.}
In this case, there are well-defined \epis \[
    \begin{array}{ccc}
     \varphi: \BS(p, p) \to \mathbb{Z} * \mathbb{Z}/p\mathbb{Z} = \langle t \rangle * \langle b \rangle: t \mapsto t, a \mapsto b&&
     \psi: \BS(p, p) \to \mathbb{Z}: t \mapsto 0, a \mapsto 1 
    \end{array}
\] Of course, these induce a canonical morphism~$\iota: \BS(p, p) \to G$, with~$\varphi$ and~$\psi$ as its components, into~$G = (\Z * \Z/p\Z) \times \Z$.
Now, both $\Z$ and $\mathbb{Z}/p\mathbb{Z}$ have the \fsqrtp. Since~$p$ is odd, we find also that~$\sqrt[2]{0} = \{0\}$ in~$\Z/p\Z$. Thus, the graph product~$G$ has the \fsqrtp by~\prref{thm:gp_fsqr}. Moreover it is easy to see that they are in fact unique.
Consequently, we are done if we can show that~$\iota$ is injective.
To this end, let~$w \in \BS(p, p)$ with~$\iota(w) = 1$ and determine its normal form \[
    \nf_S(w) = a^{k_1}t^{\eps_1}a^{k_2}t^{\eps_2} \cdots a^{k_n}t^{\eps_n}a^{k_{n+1} + pz}
\] for some~$n \in \mathbb{N}$ with~$\eps_i \in \{\pm 1\}$ and~$0 \leq k_i < p$ for~$1 \leq i \leq n + 1$ as well as~$z \in \Z$.
Now~$\varphi(w) = 1$ implies~$n = 0$ and thus~$\psi(w) = k_{n+1} + pz$.
Finally,~$\psi(w) = 0$ implies~$k_{n+1} = 0$ and~$z = 0$ and thus~$\nf_S(w) = 1$ and~$w = 1$.

\item \textbf{We have $p \neq |q|$ and both parameters are odd.}
Let~$w \in \BS(p, q)$; since conjugated elements have the same number of square roots, we may assume that~$\nf_S(w) \in \Gam^*$ is a cyclically Britton reduced word.
Now let~$x \in \sqrt[2]{w}$ be given as a Britton-reduced word.
 We claim that~$x$ is cyclically Britton-reduced as well. To see this, we write~$x \in \Gam^*$ as an alternating product \[
    x = a^{k_0}x_1a^{k_1}x_2a^{k_2}\cdots x_na^{k_n}
\] where each word~$x_i \in \Gam^*$ is Britton reduced of the form \[
    x_i = t^{\eps_i}a^{k_{i, 1}}t^{\eps_i}a^{k_{i, 2}}t^{\eps_i} \cdots t^{\eps_i}a^{k_{i, n_i}}t^{\eps_i}
\] with~$n_i > 0$ and~$\eps_i + \eps_{i+1} = 0$ for~$1 \leq i \leq n$.
Thus, we decompose~$x$ into blocks~$x_i$, whose~$t$-sequence alternate between~$\pm(1, \dots, 1)$.

       Suppose that some reduction occurs in the product~$x^2$.
Then it must apply at the boundary factor~$t^{\eps_n}a^{k_n + k_0}t^{\eps_1}$.
In particular, we have~$\eps_1 \neq \eps_n$ and thus~$n \geq 2$ is even.
Moreover, letting~$p_{-1} = p$ and~$p_{1} = q$, we must have~$p_{\eps_1} \mid k_n + k_0$. Now, suppose~$x_na^{k_n + k_0}x_1$ reduces to some Britton reduced word~$u \in \Gam^*$ with nonempty~$t$-sequence~$\ts(u)$.
Then \[
    x^2 = a^{k_0}x_1a^{k_1}\cdots x_{n-1}a^{k_{n-1}}ua^{k_1}x_2a^{k_2}\cdots x_{n}a^{k_{n}}
\] would not be cyclically Britton reduced.
But then~$\nf_S(x^2) = \nf_S(w)$ would not be cyclically Britton reduced either; a contradiction.
Thus~$u = a^{r_1}$ for some~$r_1 \in p_{-\eps_1}\Z = p_{\eps_2}\Z$.
Moreover, with the same argument, further Britton reductions must be applicable to the newly created boundary~$t^{\eps_{n-1}}a^{k_{n-1} + r_1 + a_{k_0}}t^{\eps_2}$. 
This again implies~$p_{\eps_2} \mid k_{n-1} + r_1 + k_1$ and, by the construction of~$r_1 \in p_{\eps_2}\Z$ already~$p_{\eps_2} \mid k_{n-1} + k_1$.

Continuing this way, we thus get necessary conditions~$p_{\eps_{i+1}} \mid k_{n-i} + k_i$ for~$0 \leq i \leq n$.
Now observe what happens for~$l = n/2$: We must have~$p_{\eps_{l + 1}} \mid k_{n-l} + k_l = 2k_l$.
Since~$p_{\eps_{l+1}}$ is odd by assumption, this implies~$p_{\eps_{l+1}} \mid k_l$.
But then the Britton reduced word~$x$ would have a factor~$t^{\eps_{l}}a^{k_l}t^{\eps_{l+1}}$; a contradiction.
Hence,~$x$ must be cyclically Britton reduced.

We may thus assume that~$\nf_S(x) \in \Gam^*$ is cyclically Britton reduced.
Reusing the notation of the proof of~\prref{prop:bsg_perfect}, we can now write \begin{align*}
    \nf_S(x)  = \alpha(x)a^{\beta(x)} =  a^{r_1}t^{\eps_1}a^{r_2}\cdots a^{r_m}t^{\eps_m}a^{\beta(x)}
\end{align*} and we obtain~$\nf_S(w) = \nf_S(x^2) = \alpha(x)\nf_S(a^{\beta(x)}x)$.
In particular, the word~$u = \alpha(x)$ is a uniquely determined prefix of~$\nf_S(w)$ and thus~$\alpha(y) = u$ for every~$y \in \sqrt[2]{w}$.
Moreover, from~$\nf_S(a^{\beta(y)}y) = \nf_S(a^{\beta(x)}x)$ we obtain the equation \[
    a^{\beta(y)}ua^{\beta(y)} = a^{\beta(y)}y = a^{\beta(x)}x = a^{\beta(x)}ua^{\beta(x)}
\] in~$\BS(p, q)$. With~$k = \beta(y) - \beta(x)$ this is equivalent to the equation~$a^ku = ua^{-k}$.
Now letting~$\ts(u) = (\eps_1, \dots, \eps_m)$ and~$s=p/q$, we get the condition \[
    -k = s^{\eps_m} \cdots s^{\eps_2}s^{\eps_1}k = s^E k \text{ with } E = \eps_1 + \cdots \eps_m 
\] and thus~$k = 0$, since~$|s| \neq 1$ by assumption.
Of course, this means~$\beta(y) = \beta(x)$ and finally~$|\sqrt[2]{w}| \leq 1$ for all~$w \in \BS(p, q)$.
\end{enumerate}

Additionally, observe that the proofs immediately yield the uniqueness 
of square roots. 
Therefore, we are left with two remaining cases and, for each of them, we construct 
an element $g\in \BS(p,q)$ such that $|\sqrt[2]{g}| = \infty$.

\begin{enumerate}\setcounter{enumi}{3}
\item \textbf{The parameters satisfy $1\leq p$ and~$q = -p$.}
For $n\in \N$  consider the normal form $x_n = ta^{pn}$. Then~$x_n^2 = ta^{rn}ta^{rn} = t^2$ and hence we obtain~$|\sqrt[2]{t^2}| = \infty$. 

\item \textbf{We have $1<p\leq |q|$ such that the product $pq$ is even.}
We consider the case~$1<p\in 2\N$, only. The case $1<|q|\in 2\N$ follows along the same lines. 
For $p=2r$ we define~$u = at^{-1}at$.
    Then~$x_n = u^na^ru^{-n}$ is Britton reduced since~$1\leq r<p\leq |q|$. Clearly~$x_0^2= a^p$, and 
    a straight-forward induction shows for all $n\in \N$ the following equation. 
    \[
x_{n+1}^2 = u^{n}at^{-1}ata^pt^{-1}a^{-1}ta^{-1}u^{-n} = u^{n}at^{-1}a^qta^{-1}u^{-n} = u^na^pu^{-n} = x_n^2 = a^p
    \]
Since every~$x_n$ is Britton reduced but $x_m$ and $x_n$ have distinct~$t$-sequences for $m\neq n$, the elements  $x_m$ and $x_n$ represent different elements in~$\BS(p, q)$ for $m\neq n$. Since $x_n^2 = a^p$ for all $n\in \N$ we obtain~$|\sqrt[2]{a^p}| = \infty$.
\end{enumerate}

As the above list of cases is exhaustive, this concludes the proof of the theorem.
\end{proof}

\subsection{Strongly Polycyclic and Nilpotent Groups}\label{sec:WIP}

A group $G$ is \emph{strongly polycyclic} if it possesses a subnormal series with infinite cyclic quotients.
In more detail, this means that there exist $1 = G_0 \trianglelefteq G_1 \trianglelefteq \cdots \trianglelefteq G_n = G$ such that $G_i / G_{i-1} \cong \Z$ for all $1 \leq i \leq n$.
Such groups are finitely generated, solvable, and torsion free. 

Choosing elements $g_i \in G_i$ for all $1 \leq i \leq n$ such that $g_i$ generates $G_i / G_{i-1}$, we obtain a finite generating set $\Gam = \os{g_1, \ov g_1, \dotsc, g_n, \ov g_n}$ with $\bar g=\oi g$ in $G$ and, in particular, an epimorphism $\pi : \Gam^* \to G$.
With respect to such a (strongly) \emph{polycyclic generating sequence} $g_1, \dotsc, g_n \in G$, every $g \in G$ can then be uniquely written as a product
\begin{equation}\label{eqn:polycyclic}
  g = g_1^{k_1} \cdot \ldots \cdot g_n^{k_n} \text{ with } k_1, \dotsc, k_n \in \mathbb{Z}.
\end{equation}

Reading the expression on right-hand side of \eqref{eqn:polycyclic} as a (reduced) word over $\Gam$ yields a mapping $\nf \colon G \to \Gam^\ast$, which assigns to each $g \in G$ its \emph{polycylic normal form} $\nf(g) \in \Gam^\ast$.
Collectively, these normal forms constitute the regular language
\[
  \nf(G) = (g_1^\ast \cup \bar g_1^\ast) \cdot \ldots \cdot (g_n^\ast \cup \bar g_n^\ast) \sse \Gam^\ast.
\]

\begin{lem}\label{lem:poly_perfect}
  The set of all polycyclic normal forms $\nf(G)$ is periodically perfect.
\end{lem}
\begin{proof}
  Using the above notation, every $w \in \nf(G)$ has a factor $g_i^m$ or $\bar g_i^m$ with $m \geq \abs{w} / n$ and, therefore, we have that $\exp(w) \to \infty$ as $\abs{w} \to \infty$.
\end{proof}  

It follows that any equation in $G$ has infinitely many solutions \IFF it has solutions with arbitrarily large exponent of periodicity with respect to the pair $(\pi,\nf)$.

In general, the set of square roots $\sqrt[2]{g} = \set{h \in G_i}{h^2 = g}$ need not be finite for an element $g$ of a strongly polycyclic group.
For example the \BSG $\BS(1,-1)$ is strongly polycyclic (it is the semi-direct product $\Z \rtimes \Z$ where the outer copy of $\Z$ acts by negation) and does not have the \fsqrtp by \prref{thm:bsgfsq}.
However, for an important class of strongly polycyclic groups, an even stronger property is known to hold:
\begin{prop}\label{lem:poly}
Let~$G$ be a finitely generated torsion-free nilpotent group.
Then:
\begin{enumerate}
  \item The group $G$ is strongly polycyclic.
  \item The group $G$ has unique roots, meaning that $g^n = h^n$ for $g,h \in G$ and $n \neq 0$ implies $g = h$.
\end{enumerate}
\end{prop}
\begin{proof}
  The first item follows from the center $Z(G)$ being a finitely generated \cite[1.2.16]{LennoxRobinson} torsion-free Abelian group (hence, strongly polycyclic), and the quotient $G / Z(G)$ being a finitely generated torsion-free \cite[2.3.9]{LennoxRobinson} nilpotent group of smaller nilpotency class.

  The second item follows from a similar induction: If $g^n = h^n$, then $gh^{-1} \in Z(G)$ by induction and, thus, $(gh^{-1})^n = g^nh^{-n} = 1$.
  Since $G$ is torsion free, this shows $g = h$.
\end{proof}

Nevertheless, by a result of Roman'kov~\cite[Cor.~3.2]{RomankovNilpotent2016}, it is in general undecidable for a torsion-free nilpotent group of nilpotent of class two and higher whether an input word in a generating set represents a commutator in the group.
Since `being a commutator' is clearly expressible as a quadratic equation, the satisfiability of such is also undecidable. 

\begin{rem}\label{rem:nilploy}
  As a final remark in this subsection let us note that the class of strongly polycyclic groups with unique roots is strictly larger than the class of finitely generated torsion-free nilpotent groups.
  To see this, consider the semi-direct product $(\Z\times \Z)\rtimes \Z$ where the outer group $\Z$ acts on $\Z\times \Z$ via multiplication with the matrix $\vdmatrix2111$.
\hspace*{\fill}$\diamond$
\end{rem}

\subsection{Hyperbolic Groups}\label{sec:hype}

Finally, let us consider another important class of finitely generated groups: the hyperbolic groups in the sense of Gromov~\cite{gro87}. 
Ol'shanskii~\cite{Olsh92} has shown that \emph{almost every group is hyperbolic}
in the few relator model for random groups.\footnote{In this model a random group is almost surely torsion free and 
centralizers are cyclic.}   

Roughly speaking, a hyperbolic group is a finitely generated group which, when viewed as a metric space with respect to the word metric, is negatively curved at large scales.
More precisely, $G$ is hyperbolic if there is some constant $\del$ and a finite generating set $\Gam=\oi \Gam\sse G$ such that every geodesic triangle in the corresponding Cayley graph is $\del$-thin. 
For precise definitions, we refer to Bridson and Haefliger's book~\cite{BridsonHaefliger1999}.

Free groups are hyperbolic, but $\Z^n$ for $n \geq 2$ is not.
In fact, a hyperbolic group cannot contain any \BSG as a subgroup.
This is a consequence of the following standard result on element centralizers in hyperbolic groups, which we state here for torsion-free hyperbolic groups only; for the general statement, we refer to \cite[Pro.~III.$\Gam$.3.9]{BridsonHaefliger1999}.
\begin{lem}\label{lem:hyper_ctrlzr}
  Let $G$ be a torsion-free hyperbolic group.
  For every $g \in G$ with $g \neq 1$, the centralizer $C_G(g)$ is a quasi-convex subgroup of $G$ and quasi-isometrically isomorphic to $\mathbb{Z}$.
\end{lem}

An immediate consequence of \prref{lem:hyper_ctrlzr} is the following well-known fact, which can also be proved by a reduction to free groups using so-called canonical representatives.
\begin{lem}\label{lem:hyper_roots}
  Every torsion-free hyperbolic group has unique roots.
\end{lem}

It follows from \cite{Epstein} and \cite{HoltReesJA2003} that hyperbolic groups are biautomatic with respect to the full language of geodesics.
If, as in our situation, one requires unique representatives of group elements, a canonical choice is to then use short-lex normal forms instead.
The following proposition shows that these normal forms are indeed well-behaved.

\begin{prop}\label{pro:hyper}
  Let $G$ be a torsion-free hyperbolic group with generating set $\Gam = \oi \Gam \sse G$, and let $\nf_\Gam : G \to \Gam^*$ be the associated short-lex normal forms.
  Then the following hold.
  \begin{enumerate}
  \item The set of all normal forms $\nf_\Gam(G) \sse \Gam^*$ is a regular language.
  \item The normal forms $\nf_\Gam$ are admissible in the sense of \prref{def:expinfty}.
  \end{enumerate}
\end{prop}
\begin{proof}
  For the first item see, \cite{HoltReesJA2003}.
  For the second item, let $u,p,v \in G$ with $p \neq 1$.
  We need to show that $\exp(\nf_\Gam(up^* v)) = \infty$.
  To this extent, let us consider the set 
  \[
    L = \set{\nf_\Gam(u) (\nf_\Gam(p))^n \nf_\Gam(v)}{n \in \N} \sse \Gam^*.
  \]
  Clearly, $L \sse \Gam^\ast$ is a regular language.
  Moreover, by \prref{lem:hyper_ctrlzr} and biautomaticity, each of the words $\nf_\Gam(u) (\nf_\Gam(p))^n \nf_\Gam(v) \in L$ is a $(\lambda, \varepsilon)$-quasi-geodesic representative of $up^nv \in G$ for some constants $\lambda > 0$ and $\varepsilon > 0$ independent of $n$.
  Using a recent result by Carvalho and Silva \cite[Thm.~3.8]{CarvalhoSilvaJA2026}, this implies that the set of \emph{all} geodesic representatives of the $up^nv$, 
  \[
    L' = \set{w \in \Gam^\ast}{w \text{ is a geodesic and } \pi(w) \in \pi(L) = up^*v} \sse \Gam^*,
  \]
  is a regular language.
  Hence, so is language $\nf_\Gam(up^*v) = L' \cap \nf_\Gam(G)$.
  Since this language is infinite, the pumping lemma for regular languages implies that $\exp(\nf_\Gam(up^*v)) = \infty$.
\end{proof}

\section*{Acknowledgements}
The authors thank Andr{\'e} Carvalho, Olga Kharlampovich, Paliath Narendran, and the anonymous referee for helpful discussions and comments.

\bibliographystyle{abbrv}
\bibliography{references}
\end{document}